\documentclass[11pt]{article}
\usepackage{amsfonts}
\usepackage[mathscr]{eucal}
\usepackage{mathrsfs}
\usepackage{amsmath}
\usepackage{amssymb}

\newtheorem{theorem}{Theorem}[section]

\newtheorem{proposition}[theorem]{Proposition}
\newtheorem{definition}[theorem]{Definition}
\newtheorem{remark}[theorem]{Remark}

\newcommand{\De}{\Delta}
\newcommand{\La}{\Lambda}

\newcommand{\Om}{{\Omega}}

\newcommand{\al}{{\alpha}}
\newcommand{\e}{\epsilon}
\newcommand{\ga}{{\gamma}}
\newcommand{\eps}{\varepsilon}

\newcommand{\si}{{\sigma}}

\newcommand{\om}{{\omega}}
\newcommand{\vr}{{\varrho}}

\newcommand{\Ac}{\mathfrak{A}}

\newcommand{\sC}{\mathscr{C}}

\newcommand{\sN}{\mathscr{N}}

\newcommand{\sR}{\mathscr{R}}

\newcommand{\N}{{\mathbb N}}

\newcommand{\R}{{\mathbb R}}

\newcommand{\cO}{{\cal O}}
\newcommand{\cF}{{\cal F}}

\newcommand{\cE}{{\cal E}}
\newcommand{\cH}{{\cal H}}

\newcommand{\cL}{{\cal L}}
\newcommand{\cK}{{\cal K}}

\newcommand{\cW}{{\cal W}}

\newcommand{\hf}{{\frac12}}

\newcommand{\g}{{\nabla}}

\newcommand{\bu}{\bar{u}}
\newcommand{\pd}{\partial}

\newcommand{\arr}{\rightarrow}

\newcommand{\Lto}{\widehat L_2(\Om)}
\newcommand{\Hto}{\widehat H^2_{0}(\Om)}
\newcommand{\hch}{\widehat{\cH}}
\newcommand{\wH}{\widehat{H}}

\newcommand{\tp}{\tilde{\phi}}
\newcommand{\hp}{\hat{\phi}}
\newcommand{\Ha}{{H_\alpha}}

\newcommand{\wtw}{\widetilde{w}}

\newcommand{\whw}{\widehat{w}}
\newcommand{\whW}{\widehat{W}}
\newcommand{\tV}{\widetilde{V}}

\newcommand{\di}{{\rm div\, }}
\newcommand{\wrt}{{with respect to }}

\newenvironment{declaration}[1]{\trivlist
\item[\hskip \labelsep{\bf #1 }]\ignorespaces}{\endtrivlist}
\newenvironment{proofof}[1]{\begin{declaration}{#1}}{\hfill
$\square$ \end{declaration}}
\newenvironment{proof}{\begin{proofof}{Proof.}}{\end{proofof}}
\begin{document}
\title{
Unsteady interaction
of a viscous fluid with an elastic shell modeled
by  full von Karman equations}
%A global attractor  for   a  fluid--plate interaction model}
\author{Igor Chueshov\footnote{\small e-mail:
 chueshov@univer.kharkov.ua}
 ~~~ and ~~~
  Iryna Ryzhkova\footnote{\small e-mail:
 iryonok@gmail.com} \\
 \\Department of Mechanics and Mathematics, \\
 Kharkov National University, \\ Kharkov, 61077,  Ukraine\\  }

\maketitle

\begin{abstract}
We study well-posedness and asymptotic dynamics
of a  coupled system consisting of linearized 3D
Navier--Stokes equations in a bounded domain  and a classical (nonlinear) full von Karman shallow shell equations that accounts for both transversal and lateral displacements
on a flexible   part  of the boundary.
We also take into account rotational inertia of filaments  of the shell.
Out main result shows that the problem generates a semiflow in an appropriate
phase space. The regularity provided  by  viscous dissipation
 in the fluid allows us to consider simultaneously both  cases of presence inertia in the lateral displacements and its absence.
Our second result states
 the existence of a compact global attractor for this semiflow in the case of
  presence of (rotational) damping in the transversal component
 and a particular structure of external forces.

\end{abstract}

\section{Introduction}
We consider a coupled (hybrid) system which describes an
interaction of a homogeneous viscous incompressible fluid
which  occupies a domain $\cO$ bounded by
the (solid) walls of the container $S$ and a horizontal (flat) boundary $\Om$
on which a thin (nonlinear) elastic shell is placed.
The motion of the fluid is described
by  linearized 3D Navier--Stokes equations.
To describe  deformations of the shell
 we use  the full von Karman shallow shell  model  which accounts  for both
transversal and in-plane displacements. For details concerning the shell
 model chosen we refer to  \cite{Vor1957,lagnese,lag-2} and
 also to the papers
 \cite{las-SIAM98,las-AA98,las-CPDE99,puel-AMO96,Sed91b,Sed_1997} and the references therein.
\par
This fluid-structure interaction
  model assumes that  large deflections of the shell produce
 small effect on the fluid. This  corresponds to
the case when the fluid fills the container which is large in comparison
with the size of the plate.
\par
We note that the mathematical studies of the problem
of fluid--structure interaction
in the case of viscous fluids and elastic plates/bodies
have a long history.
We refer to
\cite{CDEG05,Chu_2010,ChuRyz2011,Ggob-jmfm08,Ggob-aa09,Ggob-mmas09,Kop98}
and the references therein for the case of plates/membranes,
to \cite{CS06} in the case of moving elastic bodies,
and to \cite{avalos-amo07,aval-tri07,aval-tri09,BGLT07,BGLT08,DGHL03}
in the case of elastic bodies with the fixed interface;
see also  the literature cited in these papers.
\par
We also note that the global (asymptotic)  dynamics in nonlinear plate-fluid
models were studied before in \cite{Chu_2010,ChuRyz2011}.
The article \cite{Chu_2010} deals with
 a class of fluid-plate interaction problems, when the plate, occupying $\Om$, oscillates in {\em longitudinal} directions only.
This kind of models arises in
the study of blood flows in large arteries
(see, e.g., \cite{Ggob-jmfm08} and
the references therein).
A fluid-plate interaction model, accounting for {\em only transversal} displacement of the plate, was studied  in \cite{ChuRyz2011}.
In contrast our mathematical model  formulated below
takes into  account  {\em both}
transversal and in-plane displacements.
\smallskip\par
 Let $\cO\subset \R^3$
 be a bounded domain  with a sufficiently smooth
 boundary $\partial\cO$. We assume that $\partial\cO=\Omega\cup S$,
 where
$$
\Om\subset\{ x=(x_1;x_2;0)\, :\,x'\equiv
(x_1;x_2)\in\R^2\}
$$ with a smooth contour $\Gamma=\partial\Om$
and $S$
 is a  surface which lies in the subspace $\R^3_- =\{ x_3\le 0\}$.
 The exterior normal on $\partial\cO$ is denoted
 by $n$. We have that $n=(0;0;1)$ on $\Om$.
 \par
 We consider the following {\em linear} Navier--Stokes equations in $\cO$
for the fluid velocity field $v=v(x,t)=(v^1(x,t);v^2(x,t);v^3(x,t))$
and for the pressure $p(x,t)$:
\begin{equation}\label{fl.1}
   v_t-\nu\Delta v+\nabla p=G_f(t)\quad {\rm in\quad} \cO
   \times(0,+\infty),
\end{equation}
   \begin{equation}\label{fl.2}
   \di v=0 \quad {\rm in}\quad \cO
   \times(0,+\infty),
  \end{equation}
where $\nu>0$ is the dynamical viscosity and $G_f(t)$ is a volume force
(which may depend on $t$).
   We supplement (\ref{fl.1}) and (\ref{fl.2}) with  the (non-slip)  boundary
   conditions imposed  on the velocity field $v=v(x,t)$:
\begin{equation}\label{fl.4}
v=0 ~~ {\rm on}~S;
\quad
v\equiv(v^1;v^2;v^3)=(u^1_t;u^2_t;w_t) ~~{\rm on} ~ \Om,
\end{equation}
where $u=u(x,t)\equiv (u^1; u^2; w)(x,t)$ is the  displacement
of the shell occupying $\Om$. Here $w$ stands for transversal displacement, $\bu=(u^1;u^2)$ --- for lateral (in-plane) displacements.
\par
To describe the shell motion we use the full von Karman model which
takes into account the rotational inertia of the filaments
 and possible presence of in-plane acceleration terms
 (see the literature cited above).
 \par
 We denote by $T_f(v)$ the surface force exerted by the fluid on the shell, which is equal to $Tn\vert_\Om$, where $n$ is a outer unit normal to $\pd\cO$ at $\Om$  and $T=\{T_{ij}\}_{i,j=1}^3$ is the stress tensor of the fluid,
\[
T_{ij}\equiv T_{ij}(v)=\nu\left(v^i_{x_j}+v^j_{x_i}\right)-p\delta_{ij},
\quad i,j=1,2,3.
\]
Since  $n=(0;0;1)$ on $\Om$, we have that
\begin{equation*}%\label{surface-force}
    T_f(v)=(\nu(v^1_{x_3}+v^3_{x_1}); \nu(v^2_{x_3}+v^3_{x_2}); 2\nu \pd_{x_3}v^3 -p).
\end{equation*}
Below for some simplification we assume that for the case considered
Young's modulus $E$ and Poisson's ratio $\mu\in(0,1/2)$
are such that $Eh=2(1+\mu)$, where $h$ is the thickness of the shell.
In this case
the  (elastic) stress tensor $\{N_{ij}\}$  is given by
the formulas
\begin{align}\label{Hooke}
N_{11} =\frac 2{1-\mu}\left(\eps_{11}+\mu \eps_{22}\right),~~
N_{22} =\frac 2{1-\mu}\left(\eps_{22}+\mu \eps_{11}\right),~~ N_{12}=\eps_{12},
\end{align}
where the deformation tensor
$\{\eps_{ij}\}$ has the form
\begin{align*}
\eps_{11} &= u^1_{x_1}+ k_1 w+ \frac 12 (w_{x_1})^2, \\
\eps_{22} &=u^2_{x_2}+ k_2 w+ \frac 12 (w_{x_2})^2,\\
\eps_{12} &=u^1_{x_2}+u^2_{x_1} + w_{x_1}w_{x_2}.
\end{align*}
Here $k_1$ and $k_2$ are curvatures of the initial form of the shell which
are sufficiently smooth functions of $x'\in\Om$.
\par
After an  appropriate rescaling of the parameters and functions
we can model shell dynamics by the following equations
\begin{multline}\label{u3-eq}
M_\alpha (w_{tt} +\ga w_t)  + \De^2 w +k_1N_{11}+k_2N_{22} \\  -\pd_{x_1}(N_{11}w_{x_1}+N_{12}w_{x_2}) -
\pd_{x_2}(N_{12}w_{x_1}+N_{22}w_{x_2}) \\
=  G_3(t)- 2\nu \pd_{x_3}v^3 +p  ~~{\rm in}~~ \Omega \times (0, \infty),
\end{multline}
where  $M_\alpha=1-\alpha\De$, and
\begin{align}\label{u1u2-eq}
    \vr u_{tt}^1 & =\pd_{x_1} N_{11}+\pd_{x_2} N_{12}+ G_1(t)-\nu (v^1_{x_3}+v^3_{x_1}), \notag \\
      \vr u_{tt}^2 & =\pd_{x_1} N_{12}+\pd_{x_2} N_{22}+ G_2(t)-\nu (v^2_{x_3}+v^3_{x_2}),
\end{align}
where  $G_{sh}(t)\equiv (G_1; G_2; G_3)(t)$ is a given body force applied to  the shell.
$\alpha> 0$ and $\vr\ge 0$ are constants
 which take into account rotational inertia and in-plane inertia of
 the shell, respectively, $\ga$ is a non-negative parameter
 which describes intensity of the viscous damping of the shell material.
\par

 We impose the clamped boundary conditions on the shell
\begin{equation}
u^1|_{\pd\Om}=u^2|_{\pd\Om}=w|_{\pd\Om}=\left.\frac{\pd w}{\pd n} \right|_{\pd\Om}=0 \label{plBC}
\end{equation}
and supply \eqref{fl.1}--\eqref{plBC} with initial data
for the velocity field $v=(v^1; v^2; v^3)$ and the shell
displacement vector $u=(u^1; u^2; w)$
 of the form\footnote{
 We put the multiplier $\vr$ in the fourth relation of \eqref{IC}
 to emphasize that this relation is not needed in the case
 of negligibly small in-plane inertia ($\vr=0$).
 }
\begin{equation}
v\big|_{t=0}=v_0,~~
u\big|_{t=0}=u_0, ~~
w_t\big|_{t=0}=w_1, ~~
\vr\left[\bu_t\big|_{t=0}-\bu_1\right]=0, \label{IC}
\end{equation}
where $\bu=(u^1; u^2)$. Here $v_0=(v^1_0; v^2_0; v^3_0)$,  $u_0=(u^1_0; u^2_0; w_0)$, $w_1$,  and
$\bu_1=(u^1_1; u^2_1)$ are given vector functions which we specify
later.
\par
We  note that  \eqref{fl.2} and \eqref{fl.4} imply the following
compatibility condition
\begin{equation}\label{Com-con}
\int_\Om w_t(x',t) dx'=0 \quad \mbox{for all}~~ t\ge 0.
\end{equation}
This condition fulfills when
\[
\int_\Om w(x',t) dx'=const \quad \mbox{for all}~~ t\ge 0
\]
and can be interpreted as preservation of the volume of the fluid.
\smallskip\par
We emphasize that even in the linear case we cannot split
system \eqref{fl.1}--\eqref{IC} into two sets
of equations describing  longitudinal and
transversal plate movements separately, i.e.,
we cannot reduce the model considered to
the cases studied in \cite{Chu_2010,ChuRyz2011}.
The point is that the surface  force $T_f(v)$
 is not the sum of the corresponding loads
in the models \cite{Chu_2010} and \cite{ChuRyz2011}.
The models in \cite{Chu_2010,ChuRyz2011}  are much simpler in several respects.
For instance, in the case of  longitudinal plate deformations  only
(see \cite{Chu_2010}) the equations which correspond to \eqref{u1u2-eq}
do not contain the terms $v^3_{x_i}$ and
the model does not require any compatibility conditions like (\ref{Com-con}) because the volume of the fluid obviously preserves in the case of longitudinal deformations.
In the case of purely transversal displacements \cite{ChuRyz2011} the force exerted on the plate by the fluid contains
the pressure only.

\par
In the following remark we describe some structural properties of
the shell model chosen.

\begin{remark}\label{re:eq-u1u2}
{\rm {\bf (A)}
One can see that the equations  for the in-plane displacement vector   $\bu=(u^1;u^2)$ can be written in the vector form as follows:
\begin{equation}\label{u12-eq}
\vr \bu_{tt}+A\bu= B(w)+(G_1(t)-\nu(v^1_{x_3}+v^3_{x_1}); G_2(t)-\nu(v^2_{x_3}+v^3_{x_2})),
\end{equation}
where
the operator $A$ in \eqref{u12-eq} is defined by
\begin{equation*}
A=
-\left(
    \begin{aligned}
        & (1+\lambda)\pd^2_{x_1} +\pd^2_{x_2} & \lambda\pd_{x_1 x_2} \\
        & \lambda\pd_{x_1 x_2} &  \pd^2_{x_1} +(1+\lambda)\pd^2_{x_2} \\
    \end{aligned}
\right)
\end{equation*}
with $\lambda=\frac{1+\mu}{1-\mu}$ and $D(A)=[H^2(\Om)\bigcap H^1_0(\Om)]^2$. The nonlinear term  $B(w)$ in \eqref{u12-eq} has the  form
\begin{equation*}
B(w)=\left(
    \begin{aligned}
        &\frac1{1-\mu} \pd_{x_1}\left[2\varkappa_1 w+ (w_{x_1})^2 +\mu(w_{x_2})^2\right] +\pd_{x_2}\left[w_{x_1}w_{x_2}\right]\\
        &\pd_{x_1}\left[w_{x_1}w_{x_2}\right] +\frac1{1-\mu} \pd_{x_2}\left[2\varkappa_2 w+(w_{x_2})^2 +\mu(w_{x_1})^2\right]
    \end{aligned}
\right),
\end{equation*}
where $\varkappa_1=k_1+\mu k_2$ and  $\varkappa_2=k_2+\mu k_1$.
Thus for fixed $w$ and $v$  the equations for the in-plain displacement
$\bu$ are similar to the standard  equations of $2D$ (linear) elasticity theory.
\par
{\bf (B)} The equations in \eqref{u3-eq} and \eqref{u1u2-eq}
can be also written in the form
\begin{multline*}
(1-\alpha\De) (w_{tt} +\ga w_t)  + \De^2 w + {\rm trace}\left\{ K \sN(u)\right\}
 \\
={\rm div}\left\{ \sN(u)\nabla w\right\}
+ G_3(t)- 2\nu \pd_{x_3}v^3 +p
\end{multline*}
and
\[
    \vr \bu_{tt}  = {\rm div}\left\{ \sN(u)\right\}  +
     \left(
       \begin{array}{c}
        G_1(t)-\nu (v^1_{x_3}+v^3_{x_1}) \\
          G_2(t)-\nu (v^2_{x_3}+v^3_{x_2})\\
       \end{array}
     \right),
\]
where $K={\rm diag}\, (k_1, k_2)$ and
\[
 \sN(u)\equiv \left(
    \begin{aligned}
        N_{11} & \; N_{12} \\
        N_{12} &\; N_{22}
    \end{aligned}
\right)= \sC(\epsilon_0 (\bu)+ w K + f(\nabla w))
\]
with  $\bu=(u_1;u_2)$, $\sC(\epsilon)=2(1-\mu)^{-1}\left[\mu \, {\rm trace}\, \epsilon \cdot  I +(1-\mu)\epsilon \right]$ and
\[
\epsilon_0(\bu)=\frac 12 (\nabla \bu + \nabla ^T \bu),
\quad f(s)=\frac 12 s \otimes s, \; s\in \R^2.
\]
This form of the full von Karman system was used earlier
by many authors
in the case when the   fluid velocity field $v$ is absent and  $k_i\equiv 0$ (see, e.~g.,~\cite{puel-AMO96} or \cite{KochLa_2002}
and the references therein).
}
\end{remark}

Our main goal in this paper  is to prove well-posedness of the problem
in \eqref{fl.1}--\eqref{Com-con} in the class of finite energy solutions
(see Theorem~\ref{th:WP})
and to show  a possibility of  compact long-time dynamics
(see Theorem~\ref{th:attractor} on the existence of a compact
global attractor).
We note for the proof of uniqueness in Theorem~\ref{th:WP}
relies substantially on the $H^1$-regularity of $w_t$,
which follows from the structure of the mass operator $M_\al$
and involves Sedenko's method (see \cite{Sed91b,Sed_1997}).
To prove the existence of a global attractor in Theorem~\ref{th:attractor}
we use J.Ball's method (see \cite{ball} and \cite{MRW}).
To apply this method
we need the property $\ga>0$, i.e., assume a presence of rotational damping in the transversal component of displacement. The question whether the system under consideration demonstrates compact long-time behavior  without mechanical damping in the shell component is still open. In contrast we note that the existence of global attractors in the models considered in \cite{Chu_2010,ChuRyz2011}
{\em does not require} any mechanical damping
and compact asymptotic dynamics  of the corresponding system
is guaranteed by viscous dissipation of fluid.
One of the reasons for this is that the models in  \cite{Chu_2010,ChuRyz2011}
do not include higher order (rotational type) inertial terms
and thus the finiteness of the {\em full} dissipation integral
for the displacement follows from the viscosity of the fluid via 
the compatibility  condition in \eqref{fl.4}.

\par
The paper organized as follows.  In Section~\ref{sec:pre} we provide some
preliminary material related to Sobolev spaces and the Stokes problem.
In Section~\ref{sec:WP} we state and prove our main well-posedness result.
Section~\ref{sec:asymp} deals with long-time dynamics.

\section{Preliminaries}\label{sec:pre}
In this section we introduce Sobolev type spaces we need and
provide with some results concerning   the Stokes problem.

\subsection{Spaces and notations}
To introduce Sobolev spaces we follow approach presented
in \cite{Triebel78}.
\par
Let $D$ be a sufficiently smooth domain  and $s\in\R$.
We denote  by $H^s(D)$ the Sobolev space of order $s$
on the set $D$ which we define as  a restriction (in the sense of distributions)
 of the
space $H^s(\R^d)$ (introduced via Fourier transform).
We define the norm in  $H^s(D)$ by the relation
\[
\|u\|_{s,D}^2=\inf\left\{\|w\|_{s,\R^d}^2\, :\; w\in H^s(\R^d),~~ w=u ~~
\mbox{on}~~ D
    \right\}
\]
We also use the notation $\|\cdot \|_{D}=\|\cdot \|_{0,D}$ and $(\cdot,\cdot)_D$
for the corresponding $L_2$ norm and  inner product.
We denote by $H^s_0(D)$ the closure of $C_0^\infty(D)$ in  $H^s(D)$
(\wrt  $\|\cdot \|_{s,D}$) and introduce the spaces
\[
H^s_*(D):=\left\{ f\big|_D\, :\; f\in H^s(\R^d),\;
{\rm supp }\, f\subset \overline{D}\right\},\quad s\in \R.
\]
Below we need them to describe
boundary traces on $\Om\subset\partial \cO$.
We endow the classes $H^s_*(D)$ with the induced norms
 $\|f \|^*_{s,D}= \| f \|_{s,\R^d}$
for $f\in H^s_*(D)$. It is clear that
\[
\|f \|_{s,D}\le \|f \|^*_{s,D}, ~~ f\in H^s_*(D).
\]
However, in general the norms $\|\cdot \|_{s,D}$ and
$\|\cdot \|^*_{s,D}$ are not equivalent.
It is known that (see \cite[Theorem 4.3.2/1]{Triebel78})
that $C_0^\infty(D)$ is dense in $H^s_*(D)$ and
\begin{align*}
& H^s_*(D)\subset H^s_0(D)\subset H^s(D),~~~ s\in\R;\\
& H^s_0(D) =  H^s(D),~~~ -\infty< s\le 1/2;\\
& H^s_*(D)= H^s_0(D),~~~ -1/2<  s<\infty,~~ s-1/2\not\in
\{ 0,1,2,\ldots\}.
\end{align*}
In particular, $H^s_*(D)= H^s_0(D)= H^s(D)$ for $|s|<1/2$.
 Note that  in the notations of \cite{LiMa_1968}
the space $H^{m+1/2}_*(D)$ is the same as $H^{m+1/2}_{00}(D)$ for every
 $m= 0,1,2,\ldots$ , and for $s=m+\si$ with $0<\si<1$ we have
\[
\|u \|^*_{s,D}=\left\{ \| u\|^2_{s,D}
+\sum_{|\al|=m}\int_D \,\frac{|D^\al u(x)|^2}{d(x,\pd D)^{2\si}}\, dx
\right\}^{1/2},
\]
where $d(x,\pd D)$ is the distance between $x$ and $\pd D$.
The norm  $\|\cdot \|^*_{s,D}$ is equivalent to
 $\|\cdot \|_{s,D}$ in the case when
$s>-1/2$ and  $s-1/2\not\in\{ 0,1,2,\ldots\}$.
\par
Understanding adjoint spaces \wrt duality between
$C_0^\infty(D)$ and $[C_0^\infty(D)]'$
by Theorems 4.8.1 and 4.8.2 from \cite{Triebel78} we also have that
\begin{align*}
 [H^s_*(D)]'= H^{-s}(D),~ s\in\R, ~~~\mbox{and} ~~~
 [H^s(D)]' =  H_*^{-s} (D),~ s\in (-\infty,1/2).
\end{align*}
Below we also use the factor-spaces
$H^s(D)/\R$  with the naturally induced norm.
\medskip\par
To describe fluid velocity fields
we introduce the following spaces.
\par
Let $\mathscr{C}(\cO)$  be the class of
$C^\infty$ vector-valued solenoidal (i.e., divergence-free) functions
on $\overline{\cO}$ which vanish in a neighborhood  of $S$.
We denote by $X$ the closure of $\sC(\cO)$ \wrt  the $L_2$-norm and
by $V$ the closure  \wrt the $H^1(\cO)$-norm. One
can see that
\begin{equation}\label{X-space}
X=\left\{ v=(v^1;v^2;v^3)\in [L_2(\cO)]^3\, :\; {\rm div}\, v=0,
\gamma_n v\equiv (v,n)=0~\mbox{on}~ S\right\};
\end{equation}
and
\begin{equation*}%\label{V-space}
V=\left\{
v=(v^1;v^2;v^3)\in [H^1(\cO)]^3\, :\;
 {\rm div}\, v=0,\;
v=0~\mbox{on}~ S
  \right\}.
\end{equation*}
We equip   $X$ with $L_2$-type norm $\|\cdot\|_\cO$
and denote by $(\cdot,\cdot)_\cO$ the corresponding inner product.
The space $V$ is endowed  with the norm  $\|\cdot\|_V= \|\nabla\cdot\|_\cO$.
For some details concerning this type spaces we refer to \cite{temam-NS},
for instance.
\par
We also need the Sobolev spaces consisting of functions with zero average
on the domain $\Om$, namely
we consider the space
\[
\widehat{L}_2(\Om)=\left\{u\in L_2(\Om): \int_\Om u(x') dx' =0 \right\}
\]
and also  $\widehat H^s(\Om)=H^s(\Om)\cap\widehat L_2(\Om)$ for $s>0$
with the standard $H^s(\Om)$-norm.
The notations   $\widehat H^s_*(\Om)$ and $\widehat H^s_0(\Om)$
have a similar meaning.

To describe shell displacement  we use the spaces
\begin{equation}\label{WY-space}
W=H^1_0(\Om) \times H^1_0(\Om) \times H^2_0(\Om), \quad Y=L_2(\Om) \times L_2(\Om) \times H_\alpha,
\end{equation}
where $H_\alpha$ equals to  $\widehat{H}^1_0(\Om)$ with the equivalent norm $||\cdot||^2_\Ha=||\cdot||^2_\Om +\alpha||\g\cdot||^2_\Om$ and corresponding inner product.

\begin{remark}\label{re:hat-space}
{\rm Below we also use $\widehat{H}^2_0(\Om)$ as a state space for the
displacement of the plate.
It is clear that $\widehat{H}^2_0(\Om)$ is a closed subspace of $H^2_0(\Om)$.
We denote by $\widehat{P}$   the projection
on  $\widehat{H}^2_0(\Om)$ in $H^2_0(\Om)$ which is orthogonal
with respect to the inner product $(\Delta\cdot, \Delta\cdot)_\Om$.
One can see that
$(I-\widehat{P})H^2_0(\Om)$
  consists of functions $u\in H^2_0(\Om)$ such that $\Delta^2u=const$ and
thus has dimension one.
}
\end{remark}

\subsection{Stokes problem}

In further considerations we need some regularity  properties
of the terms responsible for fluid--plate interaction.
To this end we consider
 the following Stokes problem
\begin{align}
  -\nu\Delta v+\nabla p= g, \quad
   \di v=0 \quad {\rm in}\quad \cO; \nonumber
\\
 v=0 ~~ {\rm on}~S;
\quad
v=\psi=(\psi^1;\psi^2;\psi^3) ~~{\rm on} ~ \Om,\label{stokes}
\end{align}
where $g\in [L^2(\cO)]^3$ and $\psi\in [L^2(\Om)]^2 \times \Lto$ are given.
This type of boundary value problems for the Stokes equation
was studied by many authors  (see, e.g., \cite{lad-NSbook} and \cite{temam-NS}
and references therein). We collect some properties of solutions
to \eqref{stokes} in the following assertion.
\begin{proposition}\label{pr:stokes}
The following statements hold.
\begin{itemize}
\item [{ \bf (1)}] Let $g\in [H^{-1+\sigma}(\cO)]^3$, and $\psi\in [H^{1/2+\sigma}_*(\Om)]^3$ with
 $\int_\Om\psi^3(x')dx'=0$.
Then for every $0\le \sigma\le  1$
problem \eqref{stokes} has a unique solution
$\{v;p\}$ in $[H^{1+\sigma}(\cO)]^3\times[ H^{\sigma}(\cO)/\R]$ such that
\begin{equation}\label{stokes-bnd1}
\|v\|_{[H^{1+\sigma}(\cO)]^3}+\|p\|_{H^{\sigma}(\cO)/\R}
\le c_0\left\{\|g\|_{[H^{-1+\sigma}(\cO)]^3}+\|\psi\|_{[H_*^{\sigma+1/2}(\Om)]^3} \right\}.
\end{equation}
  \item [{ \bf (2)}]If $g=0$,  $\psi\in [H_*^{-1/2+\si}(\Om)]^3$,
$0\le \si\le 1$,  $\int_\Om\psi^3(x')dx'=0$, then
  \begin{equation}\label{stokes-bnd2}
\|v\|_{[H^{\si}(\cO)]^3}+\|p\|_{H^{-1+\si}(\cO)/\R}
\le c_0\|\psi\|_{[H_*^{-1/2+\si}(\Om)]^3}.
\end{equation}
In particular, we can define a linear operator
$N_0 : [L^2(\Om)]^2 \times \Lto \mapsto [H^{1/2}(\cO)]^3$
 by the formula
\begin{equation}\label{fl.n0}
N_0\psi=w ~~\mbox{iff}~~\left\{
\begin{array}{l}
 -\nu\Delta w+\nabla p=0, \quad
   \di w=0 \quad {\rm in}\quad \cO;
\\
 w=0 ~~ {\rm on}~S;
\quad
w=\psi ~~{\rm on} ~ \Om,
\end{array}\right.
\end{equation}
for $\psi\in [L^2(\Om)]^2 \times\Lto$
($N_0\psi$ solves \eqref{stokes} with $g\equiv 0$).
It follows from (\ref{stokes-bnd1}) and (\ref{stokes-bnd2})
that
\[
N_0 :\, [H^s_*(\Om)]^2\times \widehat{H}^s_*(\Om)\mapsto [H^{1/2+s}(\cO)]^3\cap X
~~continuously
\]
for every $-1/2\le s\le3/2$.
%  \item [{ \bf (3)}] Let $g\in [H^{-1/2+\si}(\cO)]^3$ and $\psi\in [H^{\si}_*(\Om)]^3$, with $0<\si\le 1/2$ and $\int_\Om\psi^3dx'=0$.
%Then we can define the trace of the pressure  $p$ on $\Om$,
%which possesses the property  $p|_\Om\in H^{-1+\si}(\Om)/\R$ and
%\begin{equation}\label{stokes-presure}
%\|p\|_{H^{-1+\si}(\Om)/\R} \le c_0\left\{\|g\|_{[H^{-1/2+\si}(\cO)]^3}+\|\psi\|_{[H^{\si}_*(\Om)]^3} \right\}.
%\end{equation}
\end{itemize}
\end{proposition}
\begin{proof}
We just combine proof of Proposition~2.1 \cite{Chu_2010} and Proposition~2.2 \cite{ChuRyz2011}. The argument in these references rely
on the consideration in \cite{lad-NSbook,temam-NS} and also in \cite{GSS2005}.
\end{proof}

\section{Well-Posedness Theorem}\label{sec:WP}
To define weak (variational) solutions to \eqref{fl.1}-\eqref{IC}
we need the following class $\cL_T$ of test functions $\phi$ on $\cO$:
\begin{equation*}%\label{class-L}
\cL_T=\left\{\phi \left|\begin{array}{l}
\phi\in L_2(0,T; \left[H^1(\cO)\right]^3),\; \phi_t\in L_2(0,T;  [L_2(\cO)]^3),  \\
{\rm div}\phi=0,\; \phi|_S=0,\; \phi|_\Om=b=(b^1;b^2;d),  \\
d\in  L_2(0,T; \Hto),\; b^j \in L_2(0,T; H^1_0(\Om)), \; j=1,2,\; \\
d_t\in  L_2(0,T; \widehat{H}^1_0(\Om)), \; b^j_t \in L_2(0,T; L_2(\Om)), \; j=1,2.
\end{array}\right.\right\}.
\end{equation*}
We also denote $\cL_T^0=\{\phi\in \cL_T\, :\, \phi(T)=0\}$.
\begin{definition}\label{de:solution}
{\rm
A pair of vector functions $(v(t);u(t))$ with $v=(v^1;v^2;v^3)$ and
$u=(u^1;u^2;w)$
is said to be  a weak solution to
the problem in
\eqref{fl.1}--\eqref{Com-con}  on a time interval $[0,T]$ if
\begin{itemize}
    \item $v\in L_\infty(0,T;X)\bigcap L_2(0,T; V)$;
\item $u \in L_\infty(0,T; H^1_0(\Om)\times H^1_0(\Om)\times H^2_0(\Om))$;
  \item $u_t \in L_2(0,T; \big[H^{1/2}_*(\Om)\big]^3)$ and
  the compatibility condition  $v(t)|_\Om=(u^1_t;u^2_t;w_t)(t)$ holds for almost all $t\in [0,T]$;
\item
 $(\vr u^1_t; \vr u^2_t; w_t) \in L_\infty(0,T; L_2(\Om) \times L_2(\Om)\times H_\alpha)$  and   $u(0)=u_0$;
    \item for every $\phi\in \cL_T^0$  with $\phi|_\Om=b=(b^1;b^2;d)$ the following equality holds:
        \begin{multline}
            -\!\int_0^T\!\!(v,\phi_t)_\cO dt +\nu\!\int_0^T\!\!E(v,\phi)_{\cO}  dt -\!\int_0^T\!\! \big[(M_\al w_t,d_t)_\Om+\vr (\bu_t,\bar{b}_t)_\Om\big]dt
             \\ +
            \ga\int_0^T (M_\al w_t, d)_{\Om}dt+ \!\int_0^T\!\!(\De w, \De d)_\Om dt +  \!\int_0^T\!\!a(\bu, \bar{b})_\Om dt  + \!\int_0^T\!\! q(u,b)_\Om dt \\
          =  (v_0, \phi(0))_{\cO}
          +(M_\al w_1,d(0))_\Om+\vr (\bu_1,\bar{b}(0))_\Om
          \\
          +\int_0^T(G_f(t), \phi)_{\cO} dt +\int_0^T(G_{sh}(t),b)_\Om dt , \label{weak_sol_def}
        \end{multline}
    where  $\bu=(u^1;u^2)$, $\bar{b}=(b^1;b^2)$ and also
    \begin{align}
    E(u,\phi)=&\frac 12 \sum_{i,j=1}^3\left(v^j_{x_i} + v^i_{x_j} \right)\left(\phi^j_{x_i} + \phi^i_{x_j} \right), \notag\\
    a(\bu,\bar{b})=&\sum_{i=1}^2 (\nabla u^i,\nabla b^i)_\Om +\frac{1+\mu}{1-\mu}({\rm div} u, {\rm div} b)_\Om, \label{form-a} \\
    q(u,b)=&
    (k_1 N_{11}+k_2 N_{22}, d)_\Om \notag
    \\ & +
    (N_{11}w_{x_1} + N_{12} w_{x_2}, d_{x_1})_\Om +(N_{12}w_{x_1} + N_{22} w_{x_2}, d_{x_2})_\Om
    \notag \\
    &+\frac 1{1-\mu} \left[((w_{x_1})^2+\mu (w_{x_2})^2, b^1_{x_1})_\Om + ((w_{x_2})^2+\mu (w_{x_1})^2, b^2_{x_2})_\Om \right]
     \notag\\ &
    +(w_{x_1}w_{x_2}, b^1_{x_2}+b^2_{x_1})_\Om+ \frac 2{1-\mu}
     (w, \kappa_1b^1_{x_1}+\kappa_1b^2_{x_2})_\Om. \notag
    \end{align}
   \end{itemize}
}
\end{definition}
\begin{remark}\label{re:weak-sol}
{\rm {\bf (1)}
In the case  when we neglect the inertia of longitudinal deformations
($\vr=0$) the equations in \eqref{u1u2-eq} (or in \eqref{u12-eq})
become elliptic. However we keep the initial data  for the in-plane displacement
$(u^1;u^2)$. The point is that the first order evolution for $(u^1;u^2)$
goes from the boundary  condition for the fluid velocity in \eqref{fl.4}.
\par
{\bf (2)}
 It also follows from \eqref{fl.4}  and from
the standard trace theorem that
for every weak solution $(v(t);u(t))$ we have  that
\begin{equation}\label{smooth_est}
||w(t)||^2_{H^{1/2}_*(\Om)}+||u_t^1(t)||^2_{H^{1/2}_*(\Om)}+||u_t^2(t)||^2_{H^{1/2}_*(\Om)} \le C||\g v(t)||^2_\cO
\end{equation}
for almost all $t\in [0,T]$.
This estimate does not depend on $\vr\ge0$ and provide us an additional regularity estimate for in-plane shell velocities even in the case $\vr=0$.
Below we use this observation to suggest unified way to prove well-posedness
result not distinguishing cases $\vr>0$ and $\vr=0$
(in contrast with \cite{Vor1957}, see also \cite{Sed91b} ($\vr=0$)
and \cite{Sed_1997} ($\vr>0$)).
\par
{\bf (3)} Taking in (\ref{weak_sol_def})
 $\phi(t)=\int_t^T\chi(\tau)d\tau\cdot \psi$, where $\chi$
is a smooth scalar function  and $\psi$ belongs to the space
\begin{equation}\label{space-W}
\widetilde{V}=\left\{
\psi\in  V \left|  \;
  \psi|_\Om=
\beta\equiv(\beta^1;\beta^2;\delta)\in H^1_0(\Om) \times H^1_0(\Om) \times \Hto \right. \right\},
\end{equation}
one see that the weak solution $(v(t);u(t))$  satisfies the relation
       \begin{multline}
           (v(t),\psi)_\cO
+(M_\al w_t(t),\delta)_\Om+\vr (\bu_t(t),\bar{\beta})_\Om \\
 = (v_0, \psi)_{\cO} + (M_\al w_1,\delta)_\Om+\vr (\bu_1,\bar{\beta})_\Om  \\
-\int_0^t\big[
 \nu E(v,\psi) +(\De w, \De \delta)_\Om + a(\bu, \bar{\beta})+\ga( M_\al w_t, \delta)_{\Om} \\
   + q(u, \beta)  -  (G_f, \psi)_{\cO}  -(G_{sh},\beta)_\Om\big] d\tau  \label{weak_sol_d2}
        \end{multline}
for  all $t\in [0,T]$ and
$\psi\in \widetilde{V}$  with $\psi\big\vert_\Om=\beta=(\beta^1;\beta^2;\delta)$
and $\bar{\beta}=(\beta^1;\beta^2)$.
}
\end{remark}
As a phase space we  use
\begin{equation}\label{space-cH}
\cH=
\left\{
\begin{aligned}
\left\{ (v_0;u_0;u_1)\in X\times W\times Y :\; v_0=u_1
~\mbox{on}~ \Om\right\}, \qquad \vr >0, \\
\left\{ (v_0;u_0;w_1)\in X\times W\times H_\alpha :\; (v_0)^3=w_1
~\mbox{on}~ \Om\right\}\qquad \vr=0
\end{aligned}
\right.
\end{equation}
with the norm
\[
\|(v_0;u_0;u_1)\|_\cH^2=\|v_0\|^2_{\cO}+\|u_0\|^2_{W}+\|w_1\|^2_{H_\alpha} + \vr||(u_1^1, u_1^2)||_\Om^2,
\]
where $W$ and $Y$ are given by \eqref{WY-space}.
 We also denote by $\hch$ a subspace in
$\cH$ of the form
\begin{equation*}%\label{space-cH-hat}
\hch= \left\{ (v_0;u_0;u_1)\in \cH :\; w_0\in \wH^2_0(\Om)\right\},
\end{equation*}
where $w_0$ is the third component of the initial displacement
vector $u_0$.
\par
Our main result in this section is the following well-posedness
theorem.
\begin{theorem} \label{th:WP}
Assume that $U_0=(v_0;u_0;u_1)\in \cH$, $G_f(t)\in L_2(0,T; V')$,  $G_{sh}(t)\in L_2\big(0,T; \left[H^{-1/2}(\Om)\right]^2\times H^{-1}(\Om)\big)$, $\ga \ge 0$. We also assume that $\vr\ge0$ (in the case $\vr=0$ the data $\bu_1$ are not fixed).  Then
 for any interval $[0,T]$
there exists a unique weak solution $(v(t); u(t))$ to
\eqref{fl.1}--\eqref{Com-con}
 with the initial data $U_0$. This solution possesses the following properties:
\begin{itemize}
\item In the case $\vr>0$ we have
\begin{equation}\label{cont-ws}
U(t;U_0)\equiv U(t)\equiv (v(t); u(t); u_t(t))\in C(0,T; X\times W \times Y),
\end{equation}
If $\vr=0$, then
\begin{equation}\label{cont-ws-0}
U(t;U_0)\equiv U(t)\equiv (v(t); u(t); w_t(t))\in C(0,T; X\times W \times H_\al).
\end{equation}
\item
The solutions depends continuously (both in strong and weak sense) on initial data.
More precisely,  if $\vr>0$ and $U_n\to U_0$ in the norm of $\cH$
(resp.\ weakly in $\cH$), then  $U(t;U_n)\to U(t;U_0)$ strongly (resp.\ weakly) in $\cH$
for each $t>0$.
In the case $\vr=0$  the corresponding convergence take place in $ X\times W \times H_\al$.
\item
The energy balance equality
\begin{multline}\label{lin_energy}
\cE(v(t), u(t), u_t(t))+\nu \int_0^t E(v,v) d\tau + \ga \int_0^t ||w_t||^2_{H_\alpha} d\tau \\ =\cE(v_0, u_0, u_1)
+\int_0^t(G_f,  v)_\cO d\tau +\int_0^t (G_{sh},u_t)_\Om d\tau
\end{multline}
is valid
for every $t>0$, where the energy functional $\cE$ is defined
by the relation
\begin{equation}\label{en-def}
\cE(v, u, u_t)=\frac12\left[\|v\|^2_\cO+ \|M_\al^{1/2}w_t\|^2_\Om+
\vr \|\bu_t\|^2_\Om+\| \De w\|_\Om^2 +Q(u)\right]
\end{equation}
with
\begin{align}\label{Q-def}
Q(u)=& \frac 2{(1-\mu)}\int_\Om \left( \eps_{11}^2 + \eps_{22}^2 +2\mu \eps_{11}\eps_{22} + \hf (1-\mu)\eps_{12}^2 \right)
\notag \\
= &\frac 1{2(1+\mu)}\int_\Om \left( N_{11}^2 + N_{22}^2 -2\mu N_{11}N_{22} + 2(1+\mu)N_{12}^2 \right).
\end{align}
\end{itemize}
\end{theorem}
\subsection*{Proof of Theorem~\ref{th:WP}}
We start with  the following elliptic type property of
the functional $Q(u)$.
\begin{proposition}\label{pr:Korn}
There exists a positive constant $C$ such that for every
 $u=(u^1;u^2;w)\in W= H_0^1(\Om)\times H_0^1(\Om)\times H_0^2(\Om)$
 we have that
\begin{equation}
||u^1||_{1,\Om}^2+||u^2||_{1,\Om}^2 \le C\left[ Q(u)
+\|w\|^2_\Om+\|w\|^4_{3/2,\Om}\right].
\label{Q_best}
\end{equation}
\end{proposition}
\begin{proof}
One can see that
\[
\left[u^i_{x_i}\right]^2\le 2\eps^2_{ii}+2\left[k_iw +\hf
\left(w_{x_i}\right)^2\right]^2,\quad
\left[u^1_{x_2}+u^2_{x_1} \right]^2\le 2\eps^2_{12}+
2\left[w_{x_1}w_{x_2} \right]^2.
\]
Therefore using the embedding $H^{1/2}(\Om)\subset L_4(\Om)$ we obtain
that
\[
\int_\Om \left[ \left(u^1_{x_1}\right)^2+\left(u^1_{x_1}\right)^2
+\left(u^1_{x_2}+u^2_{x_1} \right)^2\right]dx'\le C\left[ Q(u)
+\|w\|^2_\Om+\|w\|^4_{3/2,\Om}\right].
\]
Thus, property \eqref{Q_best} follows from the
Korn inequality (see, e.g., Theorem~3.1 in Chapter 3 of \cite{DuLions}).
\end{proof}

Now
we use the compactness method and  split the argument
into  several steps.
\smallskip\par\noindent
{\em Step 1. Existence of an approximate solution for the case $\vr>0$.} For
the construction of Galerkin's approximations in this case  we use
the same idea as in \cite{ChuRyz2011} (which was inspired by the method developed
in \cite{CDEG05}
for  the case of a linear plate interacting with
nonlinear Navier-Stokes equations).
\par
Let $\{\psi_i\}_{i\in \N}$ be  the orthonormal basis in $\widetilde X
=\{v\in X : v\big\vert_\Om=0\}$
consisting of the eigenvectors of the Stokes problem:
\begin{gather*}
-\De \psi_i +\nabla p_i =\mu_i\psi_i  \quad \mbox{in} \; \cO,~~~
{\rm div}\psi_i=0, \quad \psi_i|_{\pd\cO}=0.
\end{gather*}
Here  $0<\mu_1\le \mu_2\le \cdots$ are the corresponding eigenvalues.
 Denote by $\{\xi_i\}_{i\in\N}$ the basis in $\Hto$
which consists of the eigenfunctions of the following problem
\[
(\De \xi_i,\De w)_\Om=\hat{\kappa}_i (M_\alpha\xi_i, w)_\Om,~~~\forall\, w\in \wH^2_0(\Om),
\]
with the eigenvalues $0<\hat{\kappa}_1\le\hat{\kappa}_2\le\ldots$ and such that $(M_\alpha\xi_i,\xi_j)_{\Om}=\delta_{ij}$. Further, let $\{\eta_i\}_{i\in\N}$ be
 the basis in $H^1_0(\Om) \times H^1_0(\Om)$
which consists of eigenfunctions of the problem
\begin{equation*}
a(\eta_i,w)= \tilde{\kappa}_i (\eta_i, w)_\Om, \qquad
\forall\, w\in H^1_0(\Om)\times H^1_0(\Om),
\end{equation*}
with the eigenvalues $0<\tilde{\kappa}_1\le\tilde{\kappa}_2\le\ldots$ and $||\eta_i||_{\Om}=1$. The form $a(\eta,w)$ is given by \eqref{form-a}.
\par
Let  $\hp_i=N_0(0;0;\xi_i)$ and $\tp_i=N_0(\eta_i;0)$,  where the operator $N_0$
is defined by (\ref{fl.n0}).
 One can see that $\eta_i \in [H^{3/2-\delta}(\Om)]^2$. Therefore
 by Proposition~\ref{pr:stokes} we have that $\hp_i, \tp_i \in [H^{2-\delta}(\cO)]^3\cap V$ for every $\delta>0$.
\par
In what follows we suppose $\phi_i=\tp_i$, $\zeta_i=(\eta_i;0)$, $\kappa_i= \tilde{\kappa}_i$, and $\ga_i=0$;  $\phi_{n+i}=\hp_i$, $\zeta_{n+i}=(0,0,\xi_i)$, $\kappa_{n+i}=\hat{\kappa}_i$, and $\ga_{n+i}=\ga$  for $i=1,\ldots,n$. We also define the parameter
 $\vr_i$ by the relations:  $\vr_i=\vr$  for $i=1,\ldots,n$ and  $\vr_i=1$  for $i=n+1,\ldots,2n$.
 \par
We define an approximate solution as a pair of functions $(v_{n,m}; u_n)$:
\begin{gather}
v_{n,m}(t)=\sum_{i=1}^m \alpha_i(t)\psi_i +\sum_{j=1}^{2n} \dot{\beta}_j(t)\phi_j, \notag %\label{approx_sol-0}
\\  u_n(t)=\sum_{j=1}^{2n}\beta_j(t)\zeta_j+ (0;0;(I-\widehat{P})w_0), \label{approx_sol}
\end{gather}
which satisfy the relations
\begin{equation}
\dot{\alpha}_k(t) +\sum_{j=1}^{2n} \ddot{\beta}_j(t)(\phi_j,\psi_k)_\cO+\nu\mu_k \alpha_k(t)+\nu\sum_{j=1}^{2n} \dot{\beta}_j(t)E(\phi_j,\psi_k)_\cO
=(G_f, \psi_k)_\cO  \label{psi_eq}
\end{equation}
for $k=1...m$, and
\begin{multline}
\sum_{i=1}^m \dot{\alpha}_i(t)(\psi_i, \phi_k)_\cO +\sum_{j=1}^{2n} \ddot{\beta}_j(t)(\phi_j,\phi_k)_\cO+\vr_k\ddot{\beta}_k(t)    \\ +
\nu\sum_{i=1}^m  \alpha_i(t)E(\psi_i, \phi_k)_\cO +
\nu\sum_{j=1}^{2n} \dot{\beta}_j(t) E(\phi_j,\phi_k)_\cO
+\kappa_k \beta_k(t) +  \\
+\ga_k\dot{\beta}_k(t)   +q(u_n(t), \zeta_k) =
(G_f(t), \phi_k)_\cO +(G_{sh}(t), \zeta_k)_\Om  \label{phi_eq}
\end{multline}
for $k=1,\dots,2n$.
This system of ordinary differential equations is endowed
with the initial data
\[
v_{v,m}(0)=\Pi_m(v_0-N_0(0;0;w_1))+N_0(0;0;P_n w_1),
\]
\[
 u_n(0)=(R_n(u^1_0; u^2_0); P_n\widehat{P}w_0 + (I-\widehat{P})w_0), \; \dot{u}_n(0)=(R_n(u^1_1; u^2_1); P_n w_1),
\]
 where $\Pi_m$  is  the
 orthoprojector  on $Lin\{\psi_j : j=1,\ldots,m,\}$ in $\widetilde{X}$,  $P_n$
is orthoprojector on
$Lin\{\xi_i : i=1,\ldots,n\}$ in $\Lto$ and $R_n$
is orthoprojector on
$Lin\{\eta_i : i=1,\ldots,n\}$ in $L_2(\Om)\times L_2(\Om)$.
Since $\Pi_m$, $R_n$
and $P_n$  are spectral projectors we have that
\begin{equation}\label{id-conv}
(v_{v,m}(0);u_n(0);  \dot{u}_n(0))\to  (v_0;u_0;u_1)~~
\mbox{strongly in $\cH$ as $m,n\to\infty$.}
\end{equation}

\par
We can rewrite system \eqref{psi_eq} and \eqref{phi_eq} as
\begin{equation*}
M\frac{d}{dt}\left(\begin{matrix}\alpha(t) \\ \dot{\beta}(t)\end{matrix}\right) +g(\alpha(t), \beta(t), \dot{\beta}(t))=G(t)
\end{equation*}
for some locally Lipschitz  function $g\,: \R^{m+4n}\mapsto \R^{m+2n}$, where
the function $G$ lies in $L_2(0,T; \R^{m+2n})$ and
\begin{equation}\label{matr-M}
M=
\left[\begin{matrix}0&0\\0&  \sR\end{matrix}\right]+
\left[\begin{matrix} \{(\psi_i, \psi_j)_\cO\}_{j,k=1}^m & \{(\psi_l, \phi_k)_\cO\}_{l,k=1}^{m,2n} \\ \{(\phi_k, \psi_l)_\cO\}_{l,k=1}^{2n,m} &
\{ (\phi_i, \phi_j))_\cO\}_{j,k=1}^{2n}\end{matrix}\right],
\end{equation}
where $\sR={\rm diag}\{\vr_1,\ldots,\vr_{2n}\}$.
The first matrix in (\ref{matr-M}) is nonnegative and
the second one is symmetric and   strictly  positive
(since the functions $\{\psi_i, \phi_j :i=1,\ldots,m, j=1,\ldots,2n\}$ are linearly independent).
Therefore  system \eqref{psi_eq} and \eqref{phi_eq} has a unique  solution
on some time interval $[0,T']$.
\par
It follows from (\ref{approx_sol})  that
\[
v_{n,m}(t)=\sum_{i=1}^m \alpha_i(t)\psi_i + N_0[\pd_t u_n(t)],
\]
where $N_0$ is given by (\ref{fl.n0}). This implies
the following boundary compatibility condition
\begin{equation}\label{nm-comp}
v_{n,m}(t)=\pd_t u_n(t)~~ \mbox{on}~~ \Om.
\end{equation}
\par
{\em Step 2. A priori estimate
 and limit transition, $\vr>0$.}
Multiplying \eqref{psi_eq} by $\al_k(t)$ and \eqref{phi_eq} by
$\dot\beta_k(t)$, after summation we obtain an energy relation
of the form \eqref{lin_energy} for the approximate solutions
$(v_{n,m}; u_n)$
(for a similar argument we refer to
 \cite{ChuRyz2011} and also to the classical source \cite{Vor1957}  on (non-interacting)
 shell evolution).
 By Proposition~\ref{pr:Korn} and relation \eqref{smooth_est} this implies the following  a priori estimate:
 \begin{multline}\label{a-pri0}
 \sup_{t\in [0,T]}\|v_{n,m}(t)\|_\cO^2 + \\
 \sup_{t\in [0,T]}\left[\|M^{1/2}_\al \pd_t w_{n}(t)\|_\Om^2+ \vr \| \pd_t \bu_{n}(t)\|_{\Om}^2
 +\|\De w_n(t)\|^2_\Om+ \|\bu_n(t)\|_{1,\Om}^2\right]  \\
+\int_0^T\|\nabla v_{n,m}(t)\|_\Om^2 dt +  \int_0^T \| \pd_t \bu_{n}(t)\|_{[H_*^{1/2}(\Om)]^2}^2 dt
 \le C_T
 \end{multline}
 for any existence interval $[0,T]$ of approximate solutions,
 where the constant $C_T$ does not depend on $n$ and $m$.
 In particular, this implies that
  any approximate solution  can be extended on any time interval by the standard procedure, i.e., the solution is global.
  It also follows from \eqref{a-pri0} that the sequence  $\{(v_{n,m}; u_n; \pd_t u_n)\}$ contains a subsequence  such that
\begin{align}
&(v_{n,m}; u_n; \pd_t u_n) \rightharpoonup (v; u; \pd_t u) \quad \ast\mbox{-weakly in } L_\infty(0,T;\cH),\label{uv-conv}
 \\
&v_{n,m} \rightharpoonup v \quad \mbox{weakly in } L_2(0,T;V).   \label{v_conv}
\end{align}
Moreover, by the Aubin-Dubinsky  theorem
(see, e.g., \cite[Corollary~4]{sim}) we can assert that
\begin{align}
&u_n \rightarrow u \quad \mbox{strongly in } C(0,T; H^{1-\eps}_0(\Om)\times H^{1-\eps}_0(\Om)\times H^{2-\e}_0(\Om))
\label{u-strong}
\end{align}
for every $\eps>0$. Besides, the standard trace theorem yields
\begin{equation}   \label{v_conv-b}
v_{n,m} \rightharpoonup v \quad \mbox{weakly in } L_2\left(0,T; \left[H^{1/2}(\partial\cO)\right]^3\right),
\end{equation}
thus, we have
\begin{equation}\label{ut-conv}
\pd_t u_n \rightharpoonup \pd_t u \quad \mbox{weakly in } L_2(0,T; H^{1/2}_*(\Om)).
\end{equation}

One can  see  that
$(v_{n,m}; u_n; \pd_t u_n)(t)$ satisfies (\ref{weak_sol_def}) with
the test function $\phi$ of the form
\begin{equation}\label{phi-pq}
\phi=\phi_{p,q}=\sum_{i=1}^p \gamma_i(t)\psi_i
+\sum_{j=1}^q\delta_j(t)\hat\phi_j
+\sum_{j=1}^q\eta_j(t)\tilde\phi_j,
\end{equation}
where $p\le m$, $q\le n$ and $\gamma_i$, $\delta_j$, $\eta_j$
are scalar absolutely continuous functions on $[0,T]$  with time derivatives
from $L_2(0,T)$, such that $\gamma_i(T)=\delta_j(T)=\eta_j(T)=0$. Using (\ref{uv-conv}) and
(\ref{v_conv}), we can easily pass to the limit in linear terms.  Limit transition in the nonlinear terms can be done the same way  as in \cite{Vor1957}
(see also \cite{puel-AMO96}  or  \cite{KochLa_2002} for the
same type argument).
Thus, one can
show that $(v; u; \pd_t u)(t)$ satisfies (\ref{weak_sol_def})
with  $\phi=\phi_{p,q}$, where $p$ and $q$ are arbitrary.
By (\ref{id-conv}) and (\ref{u-strong}) we have $u(0)=u_0$.
The compatibility condition \eqref{fl.4} follows from  (\ref{nm-comp}),
 (\ref{ut-conv}) and (\ref{v_conv-b}).
\par
We conclude the proof of the existence of weak solutions
by showing that any function $\phi$ in $\cL_T$ can be approximated by
a sequence of functions of the form (\ref{phi-pq}), we refer to \cite{ChuRyz2011} for a similar argument in the case of zero in-plane
deformations.

This solution satisfies the energy {\em inequality}:
\begin{multline*}%\label{energy-ineq}
\cE(v(t), u(t), u_t(t))+\nu \int_0^t ||\g v||^2_\cO d\tau + \ga \int_0^t ||w_t(\tau)||^2_{H_\alpha} d\tau \\ \le \cE(v_0, u_0, u_1)
+\int_0^t(G_f(\tau),  v)_\cO d\tau +\int_0^t (G_{sh}(\tau),u_t)_\Om d\tau
\end{multline*}
for almost all  $t>0$, where the energy functional $\cE$ is defined
by \eqref{en-def}.
 \par
{\em Step 3. Existence of weak solutions in the case $\vr=0$.}
We fix some $\bu_1$ from $L_2(\Om)$ and consider the corresponding
solution $(v_\vr(t); u_\vr(t))$
with $\vr>0$. As above,
 Proposition~\ref{pr:Korn} and relation \eqref{smooth_est} imply the following  a priori estimate:
 \begin{multline}\label{a-pri-ro}
 {\rm ess}\!\!\sup_{t\in [0,T]}\left[\|v_{\vr}(t)\|_\cO^2 +
 \|M^{1/2}_\al \pd_t w_{\vr}(t)\|_\Om^2
 +\|\De w_\vr(t) \|^2_\Om+ \|\bu_\vr(t)\|_{1,\Om}^2\right]  \\
 +\int_0^T\|\nabla v_{\vr}(t)\|_\Om^2 dt +
  \int_0^T \|\pd_t \bu_{\vr}(t)\|_{[H_*^{1/2}(\Om)]^2}^2 dt
 \le C_T
 \end{multline}
 for any  interval $[0,T]$,
 where the constant $C_T$ does not depend $\vr\in (0,1)$.
This implies that
the family  $\{(v_{\vr}; u_\vr; \pd_t u_\vr)\}$ contains a subsequence  such that
\begin{align}
&(v_{\vr}; u_\vr) \rightharpoonup (v; u) \quad \ast\mbox{-weakly in } L_\infty(0,T; X\times W),\notag %\label{uv-conv-vr}
 \\
 &\pd_t w_{\vr} \rightharpoonup \pd_t w \quad \ast\mbox{-weakly in } L_\infty(0,T; H^1_0(\Om)),\notag %\label{uv-conv-vr2}
 \\
&v_{\vr} \rightharpoonup v \quad \mbox{weakly in } L_2(0,T;V), \notag  %\label{v_conv-vr}
\end{align}
when $\vr\to 0$,
and also
\begin{align*}
&v_{\vr} \rightharpoonup v \quad \mbox{weakly in } L_2(0,T; H^{1/2}(\partial\cO)),
 \notag  %\label{v_conv-b-vr}
   \\
&\pd_t u_\vr \rightharpoonup \pd_t u \quad \mbox{weakly in } L_2\left(0,T; \left[H^{1/2}(\partial\cO)\right]^3\right)\notag %\label{ut-conv-vr}.
\end{align*}
Therefore we have that
\begin{align*}
&u_\vr \rightarrow u \quad \mbox{strongly in } C(0,T; H^{1-\eps}_0(\Om)\times H^{1-\eps}_0(\Om)\times H^{2-\e}_0(\Om))
%\label{u-strong-vr}
\end{align*}
for every $\eps>0$.
This allows us to make limit transition when $\vr\to 0$ in \eqref{weak_sol_def}
and  prove the existence of weak solutions for the case $\vr=0$.
\par
{\it Step 4. Uniqueness.} We use the same idea as in \cite{Sed_1997}.
However the additional regularity  of $\bu_t$ (see Remark~\ref{re:weak-sol}(2)) makes it possible
to cover both cases $\vr>0$ and $\vr=0$ simultaneously.
\par
Let $(\widehat{v}(t);\widehat{u}(t))$ and $(\widetilde{v}(t);\widetilde{u}(t))$
 be two  solutions to the problem in question with the same initial data. These solutions satisfy \eqref{a-pri-ro} and
 their difference  $(v(t);u(t)) =(\widehat{v}(t)-\widetilde{v}(t);\widehat{u}(t)-\widetilde{u}(t))$
 possesses properties
    \begin{align*}
      &v=(v^1;v^2;v^3)\in L_\infty(0,T;X)\bigcap L_2(0,T; V);     \\
  & u =(u^1;u^2;w)\in L_\infty(0,T; H^1_0(\Om)\times H^1_0(\Om)\times H^2_0(\Om)); \\
 & w_t\in  L_\infty(0,T;  H_\alpha),~~~
 (u^1_t;  u^2_t) \in L_2(0,T; [H^{1/2}_*(\Om)]^2);
     \end{align*}
 and by \eqref{weak_sol_d2}
  satisfies the relation
       \begin{multline}
           (v(s),\psi)_\cO
+(M_\al w_t(s),\delta)_\Om+\vr (\bu_t(s),\bar{\beta})_\Om \\
 =
-\int_0^s\big[
 \nu E(v,\psi) +(\De w, \De \delta)_\Om + a(\bu, \bar{\beta}) \\ +\ga( M_\al w_t, \delta)_{\Om}
   + q(\widehat{u}, \beta) - q(\widetilde{u}, \beta)  \big] d\tau  \label{weak_sol-dif}
        \end{multline}
for  all $s\in [0,T]$ and
$\psi\in \widetilde{V}$  with $\psi\big\vert_\Om=\beta=(\beta^1;\beta^2;\delta)$
and $\bar{\beta}=(\beta^1;\beta^2)$, where $\widetilde{V}$ is defined by \eqref{space-W}.
\par
One can see that $u(t)$ lies in $H^1_0(\Om)\times H^1_0(\Om)\times H^2_0(\Om)$
for each $t\ge 0$ and
\[
\psi= i[v](s)\equiv \int_0^sv(\si)d\si\in \tV~~\mbox{with}~~ \psi\big\vert_\Om=\beta=u(s)
~~\mbox{for each $s\in [0,T]$,}
\]
where we have introduced the notation
 \begin{equation*}%\label{i-not}
 i[h](s)= \int_0^sh(\si)d\si,~~ s\in [0,T],
 \end{equation*}
 for any integrable function $h(t)$  with values in some Banach space.
\par
 Substituting this $\psi$ in \eqref{weak_sol-dif}
 for the fixed $s\in [0,T]$
 after    integration from $0$ to $t$
over $s$ we obtain that
 \begin{equation}
 \xi(t)+\vr\|\bu(t)\|_\Om^2   \le2 \left| \int_0^t ds \int_0^s d\tau \left[q(\widehat{u}(\tau), u(s))- q(\widetilde{u}(\tau), u(s))\right] \right|. \label{app_dif_est}
\end{equation}
 where
 \begin{align}
\xi(t) =& \|i[v](t) \|_\cO^2+ \|M^{1/2}_\al w(t)\|_\Om^2 +2\nu \int_0^t ds E\left(i[v](s),
i[v](s))\right) \notag  \\
& +
2\ga\int_0^t||M^{1/2}_\al w(s)||_\Om^2 ds
+\|\De i[w](t) \|_\Om^2 + a\left(i[\bu](t), i[\bu](t)\right).\notag %\label{xi-def}
\end{align}

 \par
 Now we estimate the integral term in \eqref{app_dif_est}.
 We first note that the term $q(\widehat{u}(\tau), u(s))$
 has the structure
 \[
 q(\widehat{u}(\tau), u(s))=\sum_{j=1}^4q_j(\widehat{u}(\tau), u(s)).
 \]
 Here
 \begin{itemize}
   \item $q_1(\widehat{u}, u(s))$ is a linear combination of
    terms of the form $(\varkappa_{0} \widehat{w},w(s))_\Om$ and
     $(\varkappa_{li} D_l\widehat{u}^i,w(s))_\Om$,
 where $D_l=\pd_{x_l}$ and  $\varkappa_{0}$, $\varkappa_{li}$ are smooth function, $l,i=1,2$;
   \item $q_2(\widehat{u}, u(s))$ is a linear combination of
    terms of the form
  \[
 (D_l\widehat{w}D_i\widehat{w}, w(s))_\Om,
 ~~(D_l\widehat{w}D_i\widehat{w}D_k\widehat{w}, D_mw(s))_\Om,
 ~~(\varkappa_{lm}\widehat{w} D_l\widehat{w} ,D_m w(s))_\Om;
  \]
   \item $q_3(\widehat{u}, u(s))$ is a  combination of
 the   terms  $(D_k\widehat{w},D_l\widehat{w},D_m u^i(s))_\Om$ with
  $i=1,2$;
   \item $q_4(\widehat{u}, u(s))$ consists  of
 the   terms  $(D_k\widehat{u}^i,D_l\widehat{w},D_m w(s))_\Om$ with
  $i=1,2$.
 \end{itemize}
 We also denote
 \begin{align*}
 I_m(t) &= \int_0^t ds \int_0^s d\tau \left[q_m(\widehat{u}(\tau), u(s))- q_m(\widetilde{u}(\tau), u(s))\right] \\
& = \int_0^t d\tau  \left[q_m(\widehat{u}(\tau), i[u](t)-i[u](\tau))- q_m(\widetilde{u}(\tau), i[u](t)-i[u](\tau)\right].
 \end{align*}
Below  the notation $a\sim b_i$ means that $a$ is  a linear combination of
 terms $b_i$ and $a\lesssim b_i$ means that $a$ can be estimated by  a linear combination of $b_i$.  With these notations we have
 \begin{align*}
   I_1(t)&\sim   \int_0^t ds \int_0^s d\tau \left[
   (\varkappa_{0} w(\tau),w(s))_\Om)
    +(\varkappa_{li} D_lu^j(\tau),w(s))_\Om\right] \\
    &=\int_0^t ds  \left[
   (\varkappa_{0} i[w](s),w(s))_\Om)
    +(\varkappa_{li} D_l i[{u}^j](s),w(s))_\Om\right].
 \end{align*}
This implies that
\begin{equation}\label{i1-est}
     |I_1(t)|\le C\int_0^t\xi(s)ds,~~ t\in [0,T].
\end{equation}
The most complicated term in $q_2$ is
\[
r(\widehat{w},w(s)):=
(D_l\widehat{w}D_i\widehat{w}D_k\widehat{w}, D_m w(s))_\Om.
\]
We consider
\begin{align*}
I[r]= & \int_0^t ds \int_0^s d\tau
    \left[ r(\widehat{w}(\tau),w(s))-r(\wtw(\tau),w(s))\right] \\
\sim & \int_0^t d\tau
    \left(D_lw^*(\tau)D_i w^*(\tau)D_kw(\tau),
    D_m i[w](t) -  D_m i[w](\tau) \right)_\Om,
\end{align*}
where $w^*$ is either $\widehat{w}$ or $\wtw$. Since $H^1(\Om)\subset L_p(\Om)$
for every $1\le p<\infty$
\begin{align*}
|I[r]|
\le & C \int_0^t \left[ \|\whw(\tau)\|^2_{2,\Om}+ \|\wtw(\tau)\|^2_{2,\Om}\right]
\\ & \times
\|w(\tau)\|_{1,\Om}
    \left[\| i[w](t)\|_{2,\Om} +  \|i[w](\tau)\|_{2,\Om}\right]d\tau \\
    \le & \eps \| i[w](t)\|^2_{2,\Om} + C_\eps \int_0^t\left[
\|w(\tau)\|^2_{1,\Om}+  \|i[w](\tau)\|^2_{2,\Om}\right]d\tau
\end{align*}
fore every $\eps>0$, where the constant $C_\eps$ depends on $T$
and on bounds for solutions $\whw$ and $\wtw$.
\par
Using a similar argument in other terms of $q_2$ we obtain the estimate
\begin{equation}\label{i2-est}
     |I_2(t)|\le \eps\xi(t)+ C_\eps \int_0^t\xi(s)ds,~~ t\in [0,T],~~\forall\eps>0.
\end{equation}
To estimate $I_3(t)$ we note that
\begin{align*}%\label{i3-prep}
I_3(t)
\sim  \int_0^t d\tau
    \left(D_lw^*(\tau)D_kw(\tau),
    D_m i[u^j](t) -  D_m i[u^j](\tau) \right)_\Om,
\end{align*}
where $w^*$ is either $\widehat{w}$ or $\wtw$ as above.
To estimate $I_3(t)$, we use the following Br\'esis--Gallouet  type   inequality
\begin{equation}\label{BG-ineq}
\| f  g\|_\Om\le
c_1 \left\{\log (1+\si)\right\}^{1/2}\| f\|_{\Om} \|g\|_{1,\Om}+
 \frac{c_2}{1+\si}\| f\|_{1,\Om} \|g\|_{1,\Om}
\end{equation}
for every $f,g\in H^1(\Om)$ and for all $\si\ge 0$
(this inequality can be proved the same way as Lemma~A.3.6 \cite{cl-book},
see also the Appendix in \cite{ChuShch2011}).
\par
Inequality \eqref{BG-ineq} and a priori bound \eqref{a-pri-ro} imply
\begin{align*}
\|D_lw^*D_kw\|_\Om \le &
c_1 \left\{\log (1+\si)\right\}^{1/2}\| w^*\|_{2,\Om} \|w\|_{1,\Om}+
 \frac{c_2}{1+\si}\| w^*\|_{2,\Om} \|w\|_{2,\Om}  \\
 & \le
C_1(T) \left\{\log (1+\si)\right\}^{1/2} \|w\|_{1,\Om}+
 \frac{C_2(T)}{1+\si}, ~~~\forall \si>0.
\end{align*}
Thus,
\begin{align}\label{i3-est}
|I_3(t)|\lesssim &
  \int_0^t d\tau
    \|D_lw^*(\tau)D_kw(\tau)\|_\Om
    \|D_m i[u^j](t) -  D_m i[u^j](\tau)\|_\Om \notag \\
\le &   \frac{C(T)}{1+\si} + \eps  \|i[u^j](t)\|^2_{1,\Om} \notag \\
 & +  C(\eps, T) \log (1+\si)
\int_0^t d\tau
    \left[\|w(\tau)\|^2_{1,\Om}+
    \|i[u^j](\tau)\|^2_{1,\Om}\right]   \notag \\
    \le &   \frac{C(T)}{1+\si} + \eps  \xi(t)+  C(\eps, T) \log (1+\si)
\int_0^t \xi(\tau) d\tau.
\end{align}
Now we consider $I_4(t)$. First we write it in the form
\[
I_4(t)=I_4^a(t)+ I_4^b(t),
 \]
 where
\begin{align*}
I_4^a(t)
\sim & \int_0^t d\tau
    \left(D_k u^{\ast j}(\tau)D_lw(\tau),
    D_m i[w](t) -  D_m i[w](\tau) \right)_\Om, \\
  I_4^b(t)
\sim & \int_0^t d\tau
    \left(D_k u^{j}(\tau)D_lw^*(\tau),
    D_m i[w](t) -  D_m i[w](\tau) \right)_\Om,
\end{align*}
(the star $*$ in these formulas have the same meaning as above).
Using \eqref{BG-ineq} and \eqref{a-pri-ro} we have
\begin{align}\label{i4a-est}
  |I_4^a(t)| \lesssim
 & C_T \int_0^t d\tau
     \|D_lw(\tau)\left(
    D_m i[w](t) -  D_m i[w](\tau) \right)\|_\Om \notag \\
\le &
C_T \left[\log (1+\si)\right]^{1/2}\!\int_0^t\! d\tau
     \|w(\tau)\|_{1,\Om} \| i[w](t) -   i[w](\tau)\|_{2,\Om}
     + \frac{C(T)}{1+\si}
   \notag   \\
 \le &   \frac{C(T)}{1+\si} + \eps  \xi(t)+  C(\eps, T) \log (1+\si)
\int_0^t \xi(\tau) d\tau.
\end{align}
Using integration by parts we rewrite   term $I_4^b(t)$
in the form
\begin{align*}
  I_4^b(t)
\sim &  \int_0^t d\tau
    ( u^{j}(\tau),  D_k\left[ D_lw^*(\tau)\left(
    D_m i[w](t) -  D_m i[w](\tau) \right)\right])_\Om
\end{align*}
and thus,  since $H^{1/2}(\Om)\subset L_4(\Om)$, we have
\begin{align*}
  |I_4^b(t)|\lesssim
 &  \int_0^t d\tau
    \| u^{j}(\tau)\|_{1/2, \Om} \\
    & \times
     \|D_k\left[D_lw^*(\tau)\left(
    D_m i[w](t) -  D_m i[w](\tau) \right)\right]\|_{L_{4/3}(\Om)}.
\end{align*}
One can see that
\[
\|D_k(fg)\|_{L_{4/3}(\Om)}\le C \| f\|_{1, \Om}\| g\|_{1, \Om},~~~
f,g\in H^1(\Om).
\]
Therefore
\begin{align*}
  |I_4^b(t)|\lesssim
 &  C_T \int_0^t d\tau
    \| u^{j}(\tau)\|_{1/2, \Om}   \|  i[w](t) -   i[w](\tau) \|_{2,\Om} \\
 \le & \eps\left[  \|  i[w](t)\|^2_{2,\Om} + \int_0^t d\tau
    \| u^{j}(\tau)\|^2_{1/2, \Om} \right]  \\
    & + C_{T,\eps}\left[
    \left(\int_0^t d\tau
    \| u^{j}(\tau)\|_{1/2, \Om}\right)^2 +
    \int_0^t d\tau \| i[w](\tau) \|^2_{2,\Om}\right].
\end{align*}
The trace theorem implies
\[
\int_0^t d\tau
    \| u^{j}(\tau)\|^2_{1/2, \Om}\le C \int_0^t ds E\left(i[v](s),
i[v](s))\right),
\]
and thus
\begin{align}\label{i4b-est}
  |I_4^b(t)|\le \left[ \eps + C_{T,\eps} t \right] \xi(t)
     + C_{T,\eps}    \int_0^t d\tau \xi(\tau),~~~ t\in [0,T]
\end{align}
The  estimates in \eqref{i1-est}, \eqref{i2-est}, \eqref{i3-est},
\eqref{i4a-est}, \eqref{i4b-est} and \eqref{app_dif_est} allow us
to choose $\eps$ and $T_*>0$ such that
\begin{equation*}
    \xi(t)\le  \frac{C_1}{1+\si} +  C_2 \log (1+\si)
\int_0^t \xi(\tau) d\tau~~\mbox{for all}~~t\in [0,T_*]
\end{equation*}
with arbitrary $\si>1$. Now as in \cite{Sed91b,Sed_1997}
(see also \cite[Appendix A]{cl-book})
applying Gronwall's lemma
 we can conclude that $\xi(t)\equiv 0$
for all $t\in [0,T_{**}]$ for some $T_{**}\le T_*$.

\begin{remark}\label{re:al-not}
 {\rm
 In the case $\al=0$ we can prove
 the existence of weak solutions which satisfies the energy inequality, using the same type of argument.
 The uniqueness of these solutions is still an open question.
 Sedenko's method does not work here because the nonlinearity is
 strongly supercritical when $\al=0$.
 }
\end{remark}

{\it Step 5. Continuity with respect to $t$ and the energy equality.}
First we note that the vector $(v(t);u(t);\vr u^1_t(t);\vr u^2_t(t); w_t(t))$ is weakly continuous in $\cH$  for any weak solution  $(v(t);u(t))$.
Indeed, it follows from (\ref{weak_sol_d2}) that
 $(v(t);u(t))$  satisfies the relation
       \[
           (v(t),\psi)_\cO
 = (v_0, \psi(0))_{\cO}
+\int_0^t\left[-
 \nu E(v,\psi)
          +  (G_f(\tau), \psi)_{\cO}  \right] d\tau
        \]
for almost all $t\in [0,T]$ and for all
$\psi\in V_0=\{ v\in V :\, v|_\Om=0\}\subset\widetilde{V}\subset V$,
where $\widetilde{V}$ is given by \eqref{space-W}.
This implies that $v(t)$ is weakly continuous in $V_0'$.
Since $X\subset V'_0$, we can apply  Lions lemma
(see \cite[Lemma 8.1]{LiMa_1968}) and conclude that $v(t)$ is
weakly continuous in $X$. The same lemma gives us weak continuity
 of $u(t)$ in $W$. Now using  (\ref{weak_sol_d2}) again with $\psi\in \widetilde{V}$
 we conclude that
 \begin{equation*}%\label{u-w-cont}
 t\mapsto  (M_\al w_t(t),\delta)_\Om+\vr (\bu_t(t),\bar{\beta})_\Om
 ~~\mbox{is continuous}
 \end{equation*}
 for every $\beta=(\bar\beta; \delta)\in H^1_0(\Om) \times H^1_0(\Om) \times \wH^2_0(\Om)$. This imply that
 \[
 t\mapsto   (\vr u^1_t(t); \vr u^2_t(t);w_t(t))~~\mbox{is weakly continuous}
 \] in
$Y= L_2(\Om) \times L_2(\Om) \times \wH^1_0(\Om)$
 for every $\vr\ge 0$.
\par
In the proof the energy equality we follow the scheme
of~\cite{KochLa_2002}. To this end we need to introduce finite difference operator $D_h$, depending on a small parameter $h$.

Let $g$ be a bounded function on $[0,T]$ with values in some Hilbert space. We extend $g(t)$ for all $t\in \R$ by defining $g(t)=g(0)$ for $t<0$ and $g(t)=g(T)$ for $t>T$. With this extension we denote
\begin{gather*}
g^+_h(t)=g(t+h)-g(t), \qquad g^-_h(t)=g(t)-g(t-h), \\
D_h g(t)=\frac 1{2h}(g^+_h(t)+ g^-_h(t)).
\end{gather*}
Properties of the operator $D_h$ are collected in Proposition 4.3 \cite{KochLa_2002}.
\par
Using weak continuity of weak solutions,
we can extend the variational relation in \eqref{weak_sol_def} on the class of test functions
from $\cL_T$ (instead of $\cL_T^0$) by an appropriate limit transition. More precisely, one can show that
 any weak solution $(v;u)$ (with $u=(u^1;u^2;w)\equiv (\bu;w)$)  satisfies the relation
   \begin{multline}
            -\!\int_0^T\!\!(v,\phi_t)_\cO dt +\nu\!\int_0^T\!\!E(v,\phi)_{\cO}  dt -\!\int_0^T\!\! \big[(M_\al w_t,d_t)_\Om+\vr (\bu_t,\bar{b}_t)_\Om\big] dt\\ +
            \ga\int_0^T (M_\al w_t, d)_{\Om}dt+ \!\int_0^T\!\!(\De w, \De d)_\Om dt +  \!\int_0^T\!\!a(\bu, \bar{b})_\Om dt  + \!\int_0^T\!\! q(u,b)_\Om dt \\
          =  (v_0, \phi(0))_{\cO}
          +(M_\al w_1,d(0))_\Om+\vr (\bu_1,\bar{b}(0))_\Om \\
          -\left[(v(T), \phi(T))_{\cO} +(M_\al w_t(T),d(T))_\Om+\vr (\bu_t(T),\bar{b}(T))_\Om\right]
          \\
          +\int_0^T(G_f(t), \phi)_{\cO} dt +\int_0^T(G_{sh}(t),b)_\Om dt
           \label{weak_sol_def_mod}
        \end{multline}
        for every  $\phi\in\cL_T$ with $\phi\big|_\Om=b=(b^1;b^2;d)\equiv(\bar{b};d)$.

Now  we use
\begin{equation*}
\phi=\frac 1{2h}\int_{t-h}^{t+h} v(\tau) d\tau
\end{equation*}
as a test function in \eqref{weak_sol_def_mod}. For the shell component we have test function $b=\phi|_\Om=D_h u$ -- the same one that used in \cite{KochLa_2002} for the full Karman model. Thus, all the arguments for the shell component in our model are the same as in \cite{KochLa_2002}, and we need to treat the fluid component only. Using Proposition 4.3 \cite{KochLa_2002}, one can conclude, that
\begin{multline*}
\int_0^T dt \left[ (v(t), D_h v(t))_\cO +\nu E\left(v(t), \frac 1{2h}\int_{t-h}^{t+h} v(\tau) d\tau\right) \right] \rightarrow \\
\frac 12 \left[v(T)-v(0) \right] +\nu\int_0^T dt E(v(t), v(t))
\end{multline*}
when $h\rightarrow 0$.  This makes it possible to prove the energy equality in \eqref{lin_energy}.
\par
Continuity of weak solutions with respect to $t$
stated in \eqref{cont-ws} and \eqref{cont-ws-0}
can be obtained in the standard
way from the energy equality and weak continuity (see \cite[Ch.~3]{LiMa_1968}
and also \cite{KochLa_2002}).

\par
{\it Step 6. Continuity with respect to initial data.}
First we prove the continuity \wrt weak topology.
\par
 Let $\vr>0$ and $\{U^n_0\}$ be the sequence of initial data such that $U^n_0 \rightharpoonup U_0$
  weakly in  $\cH$ (defined by \eqref{space-cH}) as $n\arr\infty$. We need to prove that
\begin{equation*}
U^n(t) \rightharpoonup U(t) \quad \mbox{weakly in}~~\cH~~\mbox{for every}~~t>0,
% \label{conv_toprove}
\end{equation*}
where $U^n(t)= (v^n(t); u^n(t); u_t^n(t))$ and $U(t)= (v(t); u(t); u_t(t))$ are weak solutions to the problem in question with the corresponding initial data.
Using  energy relation \eqref{lin_energy} and
 Proposition~\ref{pr:Korn}, we conclude that
\begin{equation*}
\sup_{[0,T]}||U^n(t)||^2_\cH \le C_T,\quad n=1,2,\ldots
\end{equation*}
This enables us to extract a subsequence such that
\begin{gather}
U^n(t) \rightharpoonup U^\ast(t) \quad \ast\mbox{-weakly in}\; L_\infty(0,T;\cH); \label{conv_sol}\\
v^n(t) \rightharpoonup v^\ast(t)\quad \mbox{ weakly in } L_2(0,T; V); \label{conv_v} \\
w^n_t(t) \rightharpoonup w^\ast_t(t) \quad \mbox{ weakly in } L_2(0,T; H_\alpha);\label{conv_w}
\end{gather}
for some $U^\ast(t)= (v^\ast(t); u^\ast(t); u^\ast_t(t))$.
Performing limit transition in \eqref{weak_sol_def}
and using weak convergence of initial data,
we obtain that $(v^\ast;u^\ast)$ is a weak solution
 and thus $U^\ast(t)=U(t)$ by the uniqueness statement.
\par
To proceed with the proof, we need to establish additional estimates for time derivatives.
\par
Taking $\psi \in \cL_T^0$ such that $\psi(0)=0$ and $\psi|_\Om=0$ as a test function in \eqref{weak_sol_def}, we obtain
\begin{equation*}
\int_0^T (v^n, \psi_t)_\cO d\tau=\int_0^T\left[ E(v^n,\psi) - (G_f, \psi)_\cO \right]d\tau,
\end{equation*}
which allows to define $v^n_t$ as a functional on $V_0=\{\phi\in V: \; \phi|_\Om=0\}$ and get the estimate
\begin{equation*}
||v^n_t||_{L_2(0,T;(V_0)')} \le C(T). %\label{dv_est}
\end{equation*}
Using $N_0b$ with $b=(b^1;b^2;d)\in C^1_0(0,T; H_0^1(\Om)\times
H_0^1(\Om)\times\times \Hto)$ as a test function, we obtain the estimate
\begin{align*}
 ||(1-\al\De)w_{tt}||_{L_2(0,T; \cF_{-2}(\Om))} +
 \vr ||\bu_{tt}||_{L_2(0,T;H^{-1}(\Om))} \le C(T),
\end{align*}
where $\cF_{-s}(\Om)$ with $s\in[1,2]$ is adjoint to $H_0^s(\Om)$
with respect to duality generated by $H_\al(\Om)$.
Thus we have the following additional convergence properties:
\begin{align*}
&v^{n}_t \rightharpoonup v_t  \qquad \mbox{weakly in } L_2(0,T; [H^{-1}(\cO)]^3), \\
&\vr \bu^{n}_{tt} \rightharpoonup \vr\bu_{tt}  \qquad \mbox{weakly in } L_2(0,T;[H^{-1}(\Om)]^2), \\
&w^{n}_{tt} \rightharpoonup w_{tt}
\qquad \mbox{weakly in } L_2(0,T;\cF_{-2}(\Om)),
\end{align*}
We note that $H_\al(\Om)\subset H_0^\sigma(\Om)\subset L_2(\Om)\subset \cF_{-2}(\Om)$
for every $\sigma\in [0,1]$. Therefore
from the relations above and also from \eqref{conv_sol}--\eqref{conv_w}
using the Aubin-Dubinsky theorem  (see \cite{sim})
 we conclude that
\begin{align*}
&v^{n} \arr v \qquad  \mbox{strongly in } C(0,T; H^{-\eps}), \\
&u^{n} \arr u  \qquad \mbox{strongly in } C(0,T;  H^{1-\eps}_0(\Om)\times H^{1-\eps}_0(\Om)\times H^{2-\eps}_0(\Om)),\\
&w^{n}_{t} \arr w_{t}  \qquad \mbox{strongly in } C(0,T;H_0^{1-\eps}(\Om)), \\
&\vr \bu^{n}_{t} \arr \vr\bu_{t}  \qquad \mbox{strongly in } C(0,T;[H^{-\eps}(\Om)]^2).
\end{align*}
In particular, this implies that $U^n(t) \arr U^\ast(t)$
weakly in $\cH$ when $U^n_0 \arr U_0$.
\par
To prove strong continuity with respect to initial data,
we use an idea borrowed from \cite{KochLa_2002}.
Due to weak continuity already established we need only to show
the  convergence of $\|U^n(t)\|_\cH$ to $\|U(t)\|_\cH$ for each $t>0$
under the condition $\| U_0^n-U_0\|_\cH\to 0$ as $n\to\infty$.
\par
Let $Q_0(u)$ be given by the first line of \eqref{Q-def}
with
\[
\eps_{11}=\eps_{11}^0\equiv u^1_{x_1},~~ \eps_{22}=\eps_{22}^0\equiv u^2_{x_2},~~
\eps_{12}=\eps_{12}^0\equiv u^1_{x_2}+u^2_{x_1}
\]
Due to the argument given in the proof of Proposition~\eqref{pr:Korn}
$\sqrt{Q_0(u)}$ is an equivalent norm on $H_0^1(\Om)\times H_0^1(\Om)$.
Thus, to establish convergence of $U^n(t)$ to $U(t)$
in $\cH$ it is sufficient to show that
\[
\cE_0(U^n(t)) \arr \cE_0(U(t))~~\mbox{as $n\to\infty$ ~~for every}~~ t>0,
\]
where
\begin{equation}\label{en-def-0}
\cE_0(v, u, u_t)=\frac12\left[\|v\|^2_\cO+ \|M_\al^{1/2}w_t\|^2_\Om+
\vr \|\bu_t\|^2_\Om+\| \De w\|_\Om^2 +Q_0(u)\right].
\end{equation}
To prove this we use the energy relation in \eqref{lin_energy}.
First we note that the potential energy term $Q(u)$
in the energy $\cE$ given in \eqref{en-def} has the form
\begin{equation}\label{q-q0}
    Q(u)=Q_0(u)+ {\rm Comp}\, (u),
\end{equation}
where ${\rm Comp}\, (u)$ is a functional  which is continuous \wrt weak
convergence in $\cH$.
Since $\cE(U^n_0) \arr \cE (U_0)$, it follows from energy equality
\eqref{lin_energy} that
\[
\lim_{n\arr\infty}\left[ \cE_0(U^n(t)) +  \Phi_t(U^n) \right] =
\cE_0(U(t)) + \nu\int_0^t E(v, v) d\tau +\Phi_t(U),
\]
where
\[
\Phi_t(U)= \nu\int_0^t E(v, v) d\tau + \ga \int_0^t (M_\alpha w_\tau, w_\tau) d\tau
\]
Using lower semicontinuity of the functional
$\Phi_t(U)$
with respect to
 weak convergence $U^n(t) \rightharpoonup U(t)$  in  $L_2(0,T;\cH)$, we obtain that
\begin{equation*}
\liminf_{n\arr\infty} \cE_0(U^n(t)) \le \cE_0(U(t)),
\end{equation*}
which together with lower semicontinuity of the (quadratic) energy $\cE_0$
 gives us the desired result. This completes the proof of Theorem~\ref{th:WP}.

\section{Stationary solutions}\label{sec:stat}
In this section following ideas presented in \cite{ChuRyz2011}
and \cite{Vorvich-stat} we
describe properties of stationary solutions in the case when
the forces $G_f$ and  $G_{sh}$ are autonomous.
These solutions are the same in both cases $\vr>0$ and $\vr=0$.
\par
It follows from Definition~\ref{de:solution} that
a stationary  solution $(v;u)\in V\times W$  satisfies the relation
       \begin{multline}
 \nu E(v,\psi) +(\De w, \De \delta)_\Om + a(\bu, \bar{\beta}) \\
   + q(u, \beta)  -  (G_f, \psi)_{\cO}  -(G_{sh},\beta)_\Om    =0
\label{stat-sol}
        \end{multline}
for  all
$\psi\in \widetilde{V}$  with $\psi\big\vert_\Om=\beta=(\beta^1;\beta^2;\delta)$
and $\bar{\beta}=(\beta^1;\beta^2)$. The space $\widetilde{V}$ is
given by \eqref{space-W}, i.e.,
\begin{equation*}
\widetilde{V}=\left\{
\psi\in  V \left|  \;
  \psi|_\Om=
\beta\equiv(\beta^1;\beta^2;\delta)\in  \widehat{W} \right. \right\},
\end{equation*}
where
\begin{equation}\label{W-hat-space}
    \widehat{W}= H^1_0(\Om) \times
H^1_0(\Om) \times \Hto.
\end{equation}
Moreover, by the comparability condition we have that $v|_{\pd\cO}=0$.
\par
We start with the following description of stationary solutions.
\par
\begin{proposition}\label{pr:stat-d2}
A couple $(v;u)\in V\times W$ is a stationary solution if and only if
 \begin{itemize}
    \item $v$ lies in  $V_0=\{\psi\in  V :   \psi|_{\pd\cO}=0\}$ and satisfies
the relation
\begin{equation}\label{v-stat}
 \nu E(v,\psi) -  (G_f, \psi)_{\cO} =0,\quad \forall\, \psi\in V_0;
\end{equation}
\item $u=(u^1;u^2;w)\in W=H^1_0(\Om) \times
H^1_0(\Om) \times H^2_0(\Om)$ satisfies the equality
       \begin{equation}
 (\De w, \De \delta)_\Om
   + \bar{q}(u, \beta)    -(G_{sh}+N_0^*G_f,\beta)_\Om    =0 \label{stat-sol_2}
        \end{equation}
for  all $\beta=(\beta^1;\beta^2;\delta)\in \whW$, where  $\widehat{W}$
is given by \eqref{W-hat-space}
and
\begin{align}\label{q-u-beta}
   \bar{q}(u,\beta)=&
   ( k_1 N_{11}+k_2 N_{22}, \delta )_\Om
    \\ & +
    (N_{11}w_{x_1} + N_{12} w_{x_2}, \delta_{x_1})_\Om +(N_{12}w_{x_1}
+ N_{22} w_{x_2}, \delta_{x_2})_\Om \notag\\
    &+  (N_{11}, \beta^1_{x_1})_\Om +(N_{12}, \beta^1_{x_2} + \beta^2_{x_1})_\Om
    +(N_{22}, \beta^2_{x_2})_\Om \notag
    \end{align}
  \end{itemize}
\end{proposition}
\begin{proof}
Taking $\beta\equiv 0$ in \eqref{stat-sol} yields \eqref{v-stat}.
Since $v\big|_{\pd\cO}=0$,
one can see that $E(v, N_0\beta)=0$ for every $\beta \in \whW$. Therefore  taking $\psi= N_0\beta$
in \eqref{stat-sol}  gives us \eqref{stat-sol_2}.
Thus any stationary solution satisfies \eqref{v-stat} and \eqref{stat-sol_2}.
Similarly, we can derive  \eqref{stat-sol} from \eqref{v-stat} and \eqref{stat-sol_2}.
\end{proof}
The following assertion shows that for the   forces
 $G_f$ and $G_{sh}$ of a special structure
we can guarantee the existence of stationary solutions.
\begin{proposition}\label{pr:stat}
We assume that
\begin{equation}\label{forces-for-stat}
G_f\equiv 0,~~ G^1_{sh}=G^2_{sh}\equiv 0~~
  and ~~G^3_{sh}\equiv g\in H^{-1}(\Om).
\end{equation}
Then any stationary solution has the
 form $(0;u)$, where $u=(u^1;u^2;w)\in W$  satisfies
      \begin{equation}\label{stat-sol_2a}
 (\De w, \De \delta)_\Om
   + \bar{q}(u, \beta)    -(g,\delta)_\Om    =0,
~~~\forall\, \beta=(\beta^1;\beta^2;\delta)\in \widehat{W}.
        \end{equation}
 Moreover, there exists at least one  solution $u=(u^1;u^2;w)$ to \eqref{stat-sol_2a} in the space $\widehat{W}$.
If in addition we assume that $k_1=k_2=0$, then
     the set of all stationary solutions  to \eqref{stat-sol_2a} from
     $\widehat{W}$ is bounded in the space $\widehat{W}$.
\end{proposition}
\begin{proof}
If $G_f\equiv 0$, then it follows from \eqref{v-stat} that $v\equiv0$. Thus
equation \eqref{stat-sol_2} turns into \eqref{stat-sol_2a}.
\par
To prove the existence of elements $u=(u^1;u^2;w)\in \whW$ satisfying \eqref{stat-sol_2a},
we note (see, e.~g.~\cite{Vorvich-stat}) that a solution can  be obtained as a minimum
point of the functional
\[
\Pi(u)=\hf\left[\|\Delta w\|^2+Q(u)\right]- (w,g)_\Om~~\mbox{on}~~\whW,
\]
where $Q$ is given by \eqref{Q-def}.
It is clear that $\Pi(u)$ is bounded from below.\footnote{This is exactly the point where
 we use the structure of the external forces assumed in \eqref{forces-for-stat}
 }
 Thus we can construct appropriate Galerkin approximations $\{u_n\}\subset\whW$
 for  the global minimum  and note that $\Pi(u_n)\le \Pi(0)$.
This facts together with Proposition~\ref{pr:Korn} provide us an a priori estimate which allows
to prove the existence of a solution  by the same method as in  \cite{Vorvich-stat}.
\par
To prove the boundedness of stationary solutions we
consider the functional  $\bar{q}(u,\beta)$ given by \eqref{q-u-beta}
with $\beta=(\beta^1;\beta^2;\delta)$, where
with $\delta=\hf w$ and $\beta^i=u^i$.
In this case we obtain that
\begin{align*}
   \bar{q}(u,\beta)=&
    -\hf(k_1 N_{11}+k_2 N_{22},w)_\Om
   +
    (N_{11}, u^1_{x_1}+ k_1w+ \hf[w_{x_1}]^2)_\Om \\ &+
    (N_{12},   u^1_{x_2} + u^2_{x_1}+ w_{x_1}w_{x_2})_\Om
    +  (N_{22}, u^2_{x_2} +  k_2w+ \hf[w_{x_2}]^2)_\Om.
    \end{align*}
Using the expressions for the deformation tensor $\{\eps_{ij}\}$ we obtain
\begin{align*}
   \bar{q}(u,\beta)=&
    -\hf(k_1 N_{11}+k_2 N_{22},w)_\Om
    \\ & +
    (N_{11}, \eps_{11})_\Om +
    (N_{12},  \eps_{12})_\Om +  (N_{22}, \eps_{22})_\Om.
    \end{align*}
The Hooke law \eqref{Hooke} yields
\begin{align*}
   \bar{q}(u,\beta)=&
    -\hf(k_1 N_{11}+k_2 N_{22},w)_\Om
    \\ & +
    \frac 2{(1-\mu)}\int_\Om \left( \eps_{11}^2 + \eps_{22}^2 +2\mu
\eps_{11}\eps_{22} + \hf (1-\mu)\eps_{12}^2 \right) \\
    =&   -\hf(k_1 N_{11}+k_2 N_{22},w)_\Om+ Q(u),
    \end{align*}
where $Q(u)$ is given by \eqref{Q-def}. Therefore under the conditions in \eqref{forces-for-stat} from
 \eqref{stat-sol_2a}  we have that
       \begin{equation*}
 \hf \|\De w\|^2_\Om + Q(u) - \hf(k_1 N_{11}+k_2 N_{22},w)_\Om
     = \hf (g,w)_\Om
%\label{stat-sol_3}
        \end{equation*}
Thus,
the set of stationary solutions is bounded provided $k_i=0$
(or even small enough).
\end{proof}
To the best of our knowledge the existence of stationary
solutions in the case of general external loads is
still an open question.
\section{Existence of a global attractor}\label{sec:asymp}
In this section we prove the existence of a compact global attractor
under the condition that the external forces
satisfy \eqref{forces-for-stat}.
\par
First we note that by
   Theorem~\ref{th:WP}  problem \eqref{fl.1}--\eqref{Com-con} generates an evolution semigroup   $S_t$ in the space $\hch$, which has the form
\begin{itemize}
  \item $\hch=X\times \whW\times Y$  in the case $\vr>0$,
  \item  $\hch= X\times \whW\times H_\alpha$  in the case $\vr=0$,
\end{itemize}
where $X$ and $Y$ are defined in \eqref{X-space} and \eqref{WY-space}
and $\whW$ is given by \eqref{W-hat-space}. The
evolution operator $S_t$ is defined as follows
\begin{itemize}
  \item {\bf Case $\vr>0$:} $S_t(v_0;u_0;u_1)\equiv U(t)=(v(t);u(t);u_t(t))$,
  where the couple $(v(t);u(t))$ solves \eqref{fl.1}--\eqref{Com-con}.
  \item  {\bf Case $\vr=0$:} $S_t(v_0;u_0;w_1)\equiv \bar{U}(t)=(v(t);u(t);w_t(t))$,
  where  $v(t)$ and $u(t)=(u^1(t);u^2(t);w(t)) $ solves \eqref{fl.1}--\eqref{Com-con} with $\vr=0$.
\end{itemize}
Our main result in this section is the following theorem.

\begin{theorem}\label{th:attractor}
  Assume that $\ga>0$ and the external forces satisfy \eqref{forces-for-stat}.
   Let the set of the stationary points in $\hch$
   of the problem  \eqref{fl.1}--\eqref{Com-con} is bounded. Then the evolution semigroup $S_t$ generated by this problem possesses a compact global attractor.
\end{theorem}
We recall
(see, e.g., \cite{BabinVishik, Chueshov,Temam})
that
\textit{global attractor}  of the  dynamical system  $(S_t, \hch)$
is defined as a bounded closed  set $\Ac\subset \hch$
which is  invariant ($S_t\Ac=\Ac$ for all $t>0$) and  uniformly  attracts
all other bounded  sets:
$$
\lim_{t\to\infty} \sup\{ {\rm dist}_{\hch}(S_ty,\Ac):\ y\in B\} = 0
\quad\mbox{for any bounded  set $B$ in $\hch$.}
$$
\begin{proof}
It follows from  energy inequality in \eqref{lin_energy} that the set
\[
\cW_R=\left\{ U\, :\; \Phi(U)\equiv \cE(U)-(g,w)_\Om\le R\right\}
\]
is forward invariant \wrt $S_t$ for each $R>0$.
Here $U=(v;u;u_t)$ with $u=(u^1;u^2;w)$ in the case $\vr>0$ and
 $U=(v;u;w_t)$  in the case $\vr=0$.
As in the proof of Proposition~\ref{pr:stat}  using
 \eqref{forces-for-stat} one can see that
$\Phi(U^n)\to+\infty$ if and only if $\|U^n\|_\cH\to+\infty$.
Therefore
 the set $\cW_R$ is bounded and any bounded set belongs
to $\cW_R$ for some $R$.
Moreover, it follows from  energy inequality \eqref{lin_energy} that
the continuous functional $\Phi(U)$ on $\hch$ possesses
the properties
(i) $\Phi\big(S_tU\big) \leq \Phi(U)$  for all $t\geq 0$ and $U\in\hch$;
(ii) the equality $\Phi(U)=\Phi(S_tU)$ holds for all $t>0$ only if $U$
is a stationary point of $S_t$.
This means that  $\Phi(U)$ is a
\textit{strict Lyapunov function} and
$(\hch,S_t)$  is a gradient dynamical system.
Therefore due to Corollary 2.29\cite{cl-mem}
(see also Theorem~4.6 in \cite{Raugel})
 we need only to prove asymptotic smoothness.
We recall that (see, e.g., \cite{Ha88} or \cite{cl-mem}) that a
dynamical system $(X,S_t)$ is said to be {\em asymptotically} smooth
if for any  closed bounded set $B\subset X$ that is positively invariant ($S_tB\subseteq B$)
one can find a compact set $\cK=\cK(B)$ which uniformly attracts $B$, i.~e.
$\sup\{{\rm dist}_X(S_ty,\cK):\ y\in B\}\to 0$ as $t\to\infty$.
\par
To prove asymptotic smoothness  we use Ball's method of  energy relations (see \cite{ball} and also \cite{MRW}).
For a convenience we
recall the abstract theorem (in a slightly relaxed form) from \cite{MRW}  which represents the main idea
of the method.
\begin{theorem}[\cite{MRW}]\label{th:ball}
Let  $S_t$ be a semigroup of  strongly
 continuous operators in some Hilbert space $\cH$.
 Assume that operators $S_t$ are also weakly continuous in $\cH$ and
 there exist  a number $\om>0$ and functionals  $\Lambda$, $L$ and $K$ on $\cH$
  such that the equality
\begin{multline}
\La(t) +\int_s^t L(\tau)e^{-2\om(t-\tau)}d\tau
=\La(s) e^{-2\om(t-s)}+\int_s^t K(\tau)e^{-2\om(t-\tau)}d\tau, \label{ball_eq}
\end{multline}
holds on the trajectories of the system $(\cH,S_t)$.
Here we use the notation $G(t)=G(S_tU)$, where $G$
is one of the symbols  $\Lambda$, $L$ and $K$.
\par
Let the
 functionals possess the properties:
\begin{itemize}
\item[{\bf (i)}] $\La:\, \cH \arr \R_+$ is a continuous bounded functional and if
 $\{U_j\}_j$ is bounded sequence in $\cH$ and  $t_j \arr +\infty$ is such that (a) $S_{t_j}U_j \rightharpoonup U$ weakly in $\cH$, and (b) $\limsup_{n\arr\infty}\La(S_{t_j}U_j) \le \La(U)$,
 then $S_{t_j}U_j \arr U$ strongly in $\cH$.

\item[{\bf (ii)}]  $K:\; \cH \arr \R$ is 'asymptotically weakly continuous' in the sense that if  $\{U_j\}_j$ is bounded in $\cH$, and  $S_{t_j}U_j \rightharpoonup U$ weakly in $\cH$ as   $t_j \arr +\infty$, then $K(S_\tau U)\in L^{loc}_1(\R_+)$ and
    \begin{equation}\label{K-conv}
        \lim_{j\arr\infty}\int_0^t e^{-2\om(t-s)}K(S_{s+t_j}U_j) ds = \int_0^t e^{-2\om(t-s)}K(S_{s}U) ds, \quad \forall t>0.
    \end{equation}
\item[{\bf (iii)}]  $L$ is 'asymptotically weakly lower semicontinuous' in the sense that if $\{U_j\}_j$ is bounded in $\cH$, $t_j \arr +\infty$, $S_{t_j}U_j \rightharpoonup U$ weakly in $\cH$, then $L(S_\tau U)\in L^{loc}_1(\R_+)$ and
    \begin{equation*}
        \liminf_{j\arr\infty}\int_0^t e^{-2\om(t-s)}L(S_{s+t_j}U_j) ds \ge \int_0^t e^{-2\om(t-s)}L(S_{s}U) ds, \quad \forall t>0.
    \end{equation*}
\end{itemize}
Then $S_t$ is asymptotically smooth.
\end{theorem}

Now we check validity of the hypotheses of Theorem~\ref{th:ball} in our case.
\smallskip\par\noindent
 {\it Step 1 (continuity):}
The continuity of the evolutional operator in both
(strong and weak) senses follows from Theorem~\ref{th:WP}.
\smallskip\par\noindent
{\it Step 2 (energy equality):}
 All the calculations below  can be justified by considering  approximate solutions.
\par
Let $\Psi(t)=\Psi(v(t), u(t), u_t(t))$, where
\[
\Psi(v, u, u_t)= (M_\al w_t,w)_\Om+
\vr (\bu_t,\bar{u})_\Om+(v,N_0[u])_\cO,
\]
where $U=(v;u;u_t)\equiv (v;\bu;w; \bu_t; w_t)$ is a weak solution and $N_0$
is given by \eqref{fl.n0}. One can see that
\begin{align*}
   \frac{d}{dt} \Psi(t)=& (M_\al w_{tt},w)_\Om+
\vr (\bu_{tt},\bar{u})_\Om+(v_t,N_0[u])_\cO
\\ &
+\|M^{1/2}_\al w_t\|^2_\Om+
\vr \|\bu_t\|^2_\Om+(v,N_0[u_t])_\cO \\
=& \|M^{1/2}_\al w_t\|^2_\Om+
\vr \|\bu_t\|^2_\Om+ (v,N_0[u_t])_\cO
-\nu E(v,N_0[u]) \\ & -\|\De w\|^2-Q(u) -
\frac1{1-\mu}\int_\Om\left[w_{x_1}^4+ w_{x_2}^4+ (1+\mu)w_{x_1}^2w_{x_2}^2 \right] dx'\\
& - \ga (M_\al w_{t},w)_\Om
+\Phi_0(t)
 + (g, w)_\Om,
\end{align*}
where $\Phi_0(t)$ is a linear combination of the terms of the form
\[
 (D_lu^i D_jw, D_mw)_\Om
 ~~\mbox{and}
 ~~(k_i w D_l w ,D_m w)_\Om.
 \]
Let $\La(t)=\cE(t)+ \eta\Psi(t)$, where $\eta>0$ will be chosen later.
Since
\[
\frac{d}{dt}\cE(t)=-\nu E(v,v) - \ga ||M_\al^{1/2}w_t||^2_{\Om}
 +(g,w_t)_\Om
\]
one can see
\begin{align*}
   \frac{d}{dt}\La(t) &+ (\ga-\eta) \|M^{1/2}_\al w_t\|^2_\Om +
   \nu E(v,v)-\eta \vr \|\bu_t\|^2_\Om
\\ & +\eta\left[\|\De w\|^2+Q(u) +
\frac1{1-\mu}\int_\Om\left[w_{x_1}^4+ w_{x_2}^4+ (1+\mu)w_{x_1}^2w_{x_2}^2 \right] dx'\right]
 \\ = &\Phi_1(t),
\end{align*}
where
\begin{align*}
\Phi_1(t) = &  \eta\left[ (v,N_0[u_t])_\cO
-\nu E(v,N_0[u]) - \ga (M_\al w_{t},w)_\Om\right]
\\ &
+\eta \Phi_0(t)
 + \eta (g, w)_\Om +(g,w_t)_\Om.
\end{align*}
Consequently
\begin{align*}
   \frac{d}{dt}\La &+ 2\om \La + (\ga-\eta-\om) \|M^{1/2}_\al w_t\|^2_\Om +
   (\nu-\om) E(v,v)-(\eta+\om) \vr \|\bu_t\|^2_\Om
\\ & +(\eta-\om)\left[\|\De w\|^2+Q(u) \right] \\ &+
\frac\eta{1-\mu}\int_\Om\left[w_{x_1}^4+ w_{x_2}^4+ (1+\mu)w_{x_1}^2w_{x_2}^2 \right] dx'
  = \Phi_1(t)+\om\eta\Psi(t).
\end{align*}
Thus we obtain  \eqref{ball_eq}
with $L$ and $K$ given by
\begin{align*}
L(t)= &(\ga-\eta-\om) \|M^{1/2}_\al w_t\|^2_\Om +
   (\nu-\om) E(v,v)
   \\ &
   -(\eta+\om) \vr \|\bu_t\|^2_\Om
    +(\eta-\om)\left[\|\De w\|^2+Q(u) \right]
    \\ & +
\frac\eta{1-\mu}\int_\Om\left[w_{x_1}^4+ w_{x_2}^4+ (1+\mu)w_{x_1}^2w_{x_2}^2 \right] dx', \\
K(t)=& \Phi_1(t)+\om\eta\Psi(t).
\end{align*}
We choose $\eta\ge\om>0$ such that $\ga-\eta-\om\ge0$ and  $\nu-\om\ge 0$.
\smallskip\par\noindent
{\it Step 3 (properties of the functionals):} Now we prove that the functionals $\La$,
 $L$ and $K$ satisfy requirements (i)-(iii) in Theorem~\ref{th:ball}.
First we  rewrite the energy $\cE$ in the form
 \[
 \cE(v,u,u_t)= \cE_0(v,u,u_t)+{\rm Comp}\, (u),
\]
where the {\em quadratic} energy functional $\cE_0$ is given by \eqref{en-def-0} and
${\rm Comp}\, (u)$ is a functional  which is continuous \wrt weak
convergence. By Proposition \ref{pr:Korn} the functional $\cE_0(v,u,u_t)$ provides an equivalent norm on $\hch$ and  therefore
 $\cE(t)$ satisfies (i) by the properties of weak convergence.
Now we show that  $\Psi$ is weakly continuous.
We start with the third term. If $u_j \rightharpoonup u$ weakly in $W$, then due to Proposition \ref{pr:stokes} $N_0 u_j \rightharpoonup N_0 u$ weakly in $[H^{3/2}(\cO)]^3\bigcap X$ and strongly in $X$. Thus, $(v_j,N_0u^j)_\cO \arr (v,N_0u)_\cO$. The remaining terms are obviously weak continuous.
Thus $\La(t)$ satisfies (i).
\par
Now we consider $K$. As above $\Psi$ and all the terms in $\Phi_1$ (except of $\Phi_0$)
are obviously  weak continuous. The same is true for $\Phi_0$ due to the fact
that $(f_1;f_2)\mapsto f_1\cdot f_2$ is a (strongly) continuous
from $H^{1/2}(\Om)\times H^{1/2}(\Om)$ into $L_2(\Om)$.
Applying this  property  to the terms  $(D_lu^i D_jw, D_mw)_\Om$ and $(k_iw D_jw, D_mw)_\Om$, we find that the functional $K$ is weakly continuous and thus
by the Lebesgue dominated convergence theorem
the convergence of integrals in \eqref{K-conv} holds.
\par
As for the functional $L$, all its terms are obviously weakly lower semicontinuous (because of the convexity norm properties, relation (\ref{q-q0}) and Proposition \ref{pr:Korn}), except of $E(v,v)$ and $- \vr \|\bu_t\|^2_\Om$. Let $\{U_j\}\subset \cH$ is bounded, $t_j\arr\infty$, and $S_{t_j}U_j\rightharpoonup U$ weakly in $\cH$. Denote by $U_j(s)$ a solution to the system under consideration with the initial data $S_{t_j}U_j$.
From the energy balance equality in \eqref{lin_energy} we have
that the sequence of the velocity field components $v_j$ of $U_j(s)$
is  bounded in $L_2(0,t;V)$. Therefore
 due to weak continuity property  of the form $E(v, v)$, we obtain
\begin{equation*}
\int_0^t e^{-\om(t-s)}E(v(s), v(s))ds \le \liminf_{j\arr\infty} \int_0^t e^{-\om(t-s)}E(v_j(s), v_j(s))ds.
\end{equation*}
Due to the standard trace theorem we can suppose that $\pd_t\bu_j=(v^1_j|_\Om, v^2_j|_\Om)$ converges to $\pd_t\bu$ weakly in $H^{1/2}_\ast(\Om))$
(hence strongly in $L_2(\Om)$) for almost all $t>0$. Therefore the Lebesgue  theorem yields that $\pd_t\bu_j \arr \pd_t\bu$ strongly in $L_2(0,T;L_2(\Om))$, so the property (iii) holds.
\par
Thus we can apply
Theorem \ref{th:ball} to prove asymptotic smoothness
and complete the proof of Theorem~\ref{th:attractor}.
\end{proof}
\begin{remark}
{\rm
  We do not know  whether  we can relax essentially the structural force hypothesis  \eqref{forces-for-stat} in Theorem~\ref{th:attractor}.
 We use \eqref{forces-for-stat} to prove the boundedness of the sublevel  set
 $\cW_R$ only. The question on the validity of this weak form of the coercivity
 of the energy functional is still open
 in the case of general external loads.
}
\end{remark}


\begin{thebibliography}{99}%{References}{}
\bibitem{avalos-amo07}
G. Avalos, The strong stability and instability of a fluid-structure semigroup,
{\it Appl. Math. Optim.}, {\bf 55} (2007), 163--184.
\bibitem{aval-tri07}
G. Avalos, R. Triggiani, The coupled PDE system arising in fluid--structure interaction.
I. Explicit semigroup generator and its
spectral properties, in: {\it Fluids and Waves},
Contemp. Math., vol. 440, AMS, Providence, RI, 2007,  15--54.
\bibitem{aval-tri09}
G. Avalos and R. Triggiani, Semigroup well-posedness in the energy space of a parabolic–
hyperbolic coupled Stokes--Lam\'e PDE system of fluid-structure interaction,
{\it Discr. Contin.
Dyn. Sys.}, Ser.S, {\bf 2} (2009), 417--447.


%\bibitem{AHKM03}
%T. Azizov, V. Hardt, N. Kopachevsky, R. Mennicken,
%On the problem of small motions and normal
%oscillations of a viscous fluid
%in a partially filled container,
%{\it Math. Nachr.}  {\bf 248/249} ( 2003),  3--39.

   \bibitem{BabinVishik} A.V. Babin, M.I. Vishik,
    \textit{Attractors of Evolution Equations.} North-Holland, Amsterdam, 1992.


\bibitem{ball}  J. Ball, Global attractors for semilinear wave
equations,
 Discr. Cont. Dyn. Sys. 10 (2004),   31--52.



\bibitem{BGLT07}
V. Barbu, Z. Gruji\'c, I. Lasiecka, A. Tuffaha,
Existence of the energy-level weak solutions for a nonlinear fluid--structure
interaction model, in: {\it Fluids and Waves},
Contemp. Math., vol. 440, AMS, Providence, RI, 2007,  55--82.
\bibitem{BGLT08}
V. Barbu, Z. Gruji\'c, I. Lasiecka, A. Tuffaha,
Smoothness of weak solutions to a nonlinear fluid--structure interaction model,
{\it Indiana Univ. Math. J.} {\bf 57} (2008), 1173--207.


\bibitem{CDEG05}
A. Chambolle, B. Desjardins, M. Esteban,  C. Grandmont,
 Existence of weak solutions for the unsteady interaction
of a viscous fluid with an elastic plate.
{\it J. Math. Fluid Mech.} \textbf{7} (2005), 368--404.


\bibitem{Chu_2010} I. Chueshov, A global attractor for a fluid-plate interaction
model accounting only for longitudinal
deformations of the plate,
{\it Math. Meth. Appl. Sci.}, {\bf 34} (2011), 1801--1812.


    \bibitem{Chueshov} I. Chueshov, \textit{Introduction to the
Theory of Infinite-Dimensional Dissipative Systems}. Acta, Kharkov,
1999 (in Russian); English translation: Acta, Kharkov, 2002;
$http://www.emis.de/monographs/Chueshov/$.




\bibitem{cl-mem}
I. Chueshov and I. Lasiecka,
\textit{Long-Time Behavior of Second Order Evolution  Equations
with  Nonlinear Damping},
Memoirs of AMS, vol.195, no. 912, AMS, Providence, RI,
2008.

\bibitem{cl-book}
I. Chueshov and I. Lasiecka, {\it Von Karman  Evolution Equations},
Sprin\-ger, New York, 2010.

\bibitem{ChuRyz2011} I. Chueshov, I. Ryzhkova,  A global attractor for a
fluid-plate interaction model,   Preprint ArXiv:1109.4324v1
 (September 2011).

\bibitem{ChuShch2011} I. Chueshov, A. Shcherbina,
Semi-weak well-posedness and  attractor for
  2D Schr\"odinger-Boussinesq equations, submitted.


\bibitem{CS06}
 D. Coutand, S. Shkoller, Motion of an elastic solid inside an
incompressible viscous fluid, {\it Arch. Ration. Mech. Anal.}  {\bf 176}
(2005), 25--102.

\bibitem{DGHL03}
 Q. Du, M.D. Gunzburger, L.S. Hou, J. Lee,
 Analysis of a linear fluid--structure interaction problem,
 {\it Discrete Contin. Dyn.
Syst.} {\bf  9} (2003), 633--650.

\bibitem{DuLions}
G. Duvaut, J. L. Lions,
{\it Inequalities in Mechanics and Physics}. Springer
Berlin, New York, 1976.
\bibitem{GSS2005}
G.Galdi, C. Simader,  H. Sohr, A class of solutions to
stationary Stokes and Navier-Stokes equations with boundary data in
$W^{-1/q,q}$, {\it Math. Annalen}  {\bf 331} (2005), 41--74.

\bibitem{Ggob-jmfm08}
 M. Grobbelaar-Van Dalsen, On a  fluid-structure model in which
the dynamics of the structure involves the shear
stress due to the fluid, {\it J. Math. Fluid Mech.} {\bf 10}  (2008), 388--401.
\bibitem{Ggob-aa09}
 M. Grobbelaar-Van Dalsen,
 A new approach to the stabilization of a fluid-structure
interaction model, {\it Appl. Anal.} {\bf 88} (2009),  1053--1065.
 \bibitem{Ggob-mmas09}
 M. Grobbelaar-Van Dalsen,
Strong stability for a fluid-structure model, {\it Math. Methods
 Appl. Sci.} {\bf 32} (2009),  1452--1466.
\bibitem{Ha88}
J.K.  Hale,  {\it Asymptotic Behavior of  Dissipative
Systems}. Amer. Math. Soc., Providence, RI, 1988.

%\bibitem{kellogg}
%B. Kellogg, Properties of solutions of elliptic boundary value problems,
%in: {\it The Mathematical
% Foundations of the Finite Element Method with
%Applications to Partial Differential Equations}, A.K. Aziz (Ed.),
% Academic Press, New York, 1972, pp. 47--81.
%%%%%%%%%%%%%%%%%%%%%%%%%
\bibitem{KochLa_2002} H. Koch and I. Lasiecka. Hadamard well-posedness of weak solutions in nonlinear dynamic elasticity-full von Karman systems, in:
    {\em Prog. Nonlinear Differ. Equ. Appl. vol.50}, pp.197-216.
     Basel: Birkh\"auser, 2002.

\bibitem{Kop98} N. Kopachevskii, Yu. Pashkova,  Small oscillations
of a viscous  fluid  in a vessel bounded by an elastic membrane,
\emph{Russian J. Math. Phys.} \textbf{5} (1998), no.4, 459--472.
\bibitem{lad-NSbook}
O. Ladyzhenskaya, {\it Mathematical Theory of Viscous Incompressible Flow}, GIFML, Moscow, 1961 (1st Russian edition); Nauka, Moscow, 1970 (2nd Russian
edition); Gordon and Breach, New York, 1963 and 1969 (English translations of
the 1st Russian edition).

\bibitem{lagnese}
 J. Lagnese, {\it Boundary Stabilization of Thin Plates},
SIAM, Philadelphia, 1989.
\bibitem{lag-2}
J. Lagnese,
Modeling and stabilization of nonlinear plates,
\textit{Int. Ser. Num. Math.}, 100 (1991), 247--264.
%%%%
%\bibitem{lag-lions}
% J. Lagnese and  J.L.Lions, {\it Modeling, Analysis and Control of Thin
% Plates}, Masson, Paris, 1988.
%\bibitem{lt-book} I. Lasiecka and R. Triggiani,
%{\it Control Theory for Partial Differential Equations},
%Cambridge University Press,  Cambridge, 2000.


\bibitem{las-SIAM98}
I. Lasiecka,
 Uniform stabilizability of a full von Karman system
with nonlinear boundary feedback, {\em SIAM J. Control Optim.},
 {\bf 36} (1998), 1376--1422.
\bibitem{las-AA98}
I. Lasiecka, Weak, classical and intermediate solutions to full von Karman system of dynamic nonlinear elasticity, {\em Appl. Anal.}, {\bf 68} (1998), 121--145.

\bibitem{las-CPDE99} I.Lasiecka, Uniform decay rates for full von
Karman system of dynamic thermoelasticity with free boundary
conditions and partial boundary dissipation,  {\em Comm.
Partial Diff. Eqs.}, {\bf 24} (1999), 1801--1849.


\bibitem{LiMa_1968} J.-L. Lions, E. Magenes,
{\it Probl\'emes aux limites non homog\'enes et applications,} Vol. 1,
Dunod, Paris, 1968.

%\bibitem{Lions_1969} J.-L. Lions, {\it Quelques methodes de resolution des probl\'emes aux limites non lineaire},
%Dunod, Paris, 1969.

%\bibitem{pazy} A. Pazy,
%{\it Semigroups of Linear Operators and Applications to Partial
%Differential Equations}, Springer, New York, 1986.



\bibitem{MRW}
 I. Moise, R.  Rosa, and X. Wang, Attractors for non-compact
semigroups via energy equations,  {\em Nonlinearity}, {\bf 11} (1998), 1369--1393.



\bibitem{puel-AMO96}
J. Puel and M. Tucsnak, Global existence for the full von Karman system,
{\em App. Math. Optim.} {\bf 34} (1996),  139--160.

\bibitem{Raugel} G.
Raugel, Global attractors in partial differential
equations, in: \textit{Handbook of Dynamical Systems, vol. 2}, pp.885-992. Elsevier  Sciences, Amsterdam, 2002.
\bibitem{Sed91b}
 V.I.Sedenko, The uniqueness theorem of generalized solution of
initial boundary value problem of nonlinear oscillations theory of
shallow shells with small inertia of longitudinal displacements,
\textit{ Izvestiya AN SSSR, mehanika tverdogo tela},  6 (1991),
142--150, in Russian.


\bibitem{Sed_1997}
V.~I.~Sedenko, On the uniqueness theorem for generalized solutions of initial-boundary problems for the Marguerre--Vlasov vibrations of shallow shells with clamped boundary conditions, {\em Appl. Math. Optim.}, {\bf 39} (1999), 309--326

\bibitem{sim} J. Simon, Compact sets in the space $L^p(0,T;B)$,
{\em Annali di Matematica
 Pura ed Applicata, Ser.4}  \textbf{148} (1987), 65--96.



 \bibitem{Temam} R. Temam,
\textit{Infinite-Dimensional Dynamical Dystems in Mechanics and
Physics}, Springer, New York, 1988.
\bibitem{temam-NS}
R. Temam, {\it Navier-Stokes Equations: Theory and Numerical Analysis}, Reprint of the 1984
edition, AMS Chelsea Publishing, Providence, RI, 2001.

\bibitem{Triebel78}
H. Triebel, {\it Interpolation Theory, Functional Spaces and
Differential Operators}, North Holland, Amsterdam, 1978.

\bibitem{Vorvich-stat} I.~I.~Vorovich, On the existence of solutions  in the nonlinear
    theory  of shells, {\em
Izvestiya AN SSSR, Matematika} {\bf 19}(4) (1955), 173--186, in Russian.

\bibitem{Vor1957} I.~I.~Vorovich, On some direct methods in nonlinear oscillations of shallow shells, {\em
Izvestiya AN SSSR, Matematika} {\bf 21}(6) (1957), 142--150, in Russian.
\end{thebibliography}
\end{document}